\documentclass[12pt,a4paper,reqno]{amsart} 
\usepackage[utf8]{inputenc}

\usepackage{lmodern}

\usepackage{a4wide}
\usepackage{fancyhdr}
\usepackage[english]{babel}

\usepackage{amsmath, mathtools}
\usepackage{amsfonts}
\usepackage{amssymb}
\usepackage{amsthm}
\usepackage[dvipsnames]{xcolor}
\usepackage{tikz}
\usepackage{pgfplots}
\usetikzlibrary{positioning}
\usetikzlibrary{intersections}
\usetikzlibrary{calc}
\usepackage{csquotes}
\usepackage{mathabx}
\usepackage{bm}
\usepackage{todonotes}
\usepackage{hyperref}
\usepackage{cleveref}
\usepackage{bbm}

\theoremstyle{definition}

\numberwithin{equation}{section}

\title[Soap Film Bridge Subjected to an Electrostatic Force: The Balanced Case]{Stationary Soap Film Bridge Subjected to an Electrostatic Force: The Balanced Case}
\author{Lina Sophie Schmitz}
\date{September 27, 2024}
\address{Institut f\"ur Angewandte Mathematik, Leibniz Universit\"at Hannover, Welfengarten 1 \\ D-30167 Hannover, Germany}
\email{schmitz@ifam.uni-hannover.de}

\usepackage{graphicx}

\newtheoremstyle{common}
	{}    
	{}    
    {\itshape}
    {0em}
    {\bfseries}
    {}
    {.5em}
    {}
\theoremstyle{common}

\numberwithin{subsection}{section}

\numberwithin{figure}{section}

\newtheorem{thm}{{\bf Theorem}}
\numberwithin{thm}{section}

\newtheorem{lem}[thm]{{\bf Lemma}}

\newtheorem{prop}[thm]{{\bf Proposition}}

\numberwithin{equation}{section}

\newtheoremstyle{commondef}  
    {6pt}
    {6pt}
    {}
    {0em}
    {\bfseries}
    {}
    {.5em}
    {}
\theoremstyle{commondef}

\newtheorem{bem}[thm]{{\bf Remark}}

\newtheorem{defin}[thm]{{\bf Definition}}

\renewenvironment{proof}{{\bf Proof}.}{\qed\\}

\begin{document}

\begin{abstract} 
We analytically study a stationary free boundary problem describing a soap film bridge subjected to an electrostatic force. Starting from a cylinder that the soap film forms if its surface tension and the electrostatic force are perfectly balanced, we construct a local branch of stationary solutions. Then, we prove that there is a sharp threshold value for a characteristic model parameter at which the stability behaviour of this branch switches from stable to unstable. Finally, we show that the soap film deflects monotonically outwards if the strength of the electrostatic force is increased. Besides rigorous results, we also discuss a possible balancing effect of the electrostatic force. 

\end{abstract}

\maketitle

\allowdisplaybreaks


\section{Introduction}
A classic example for the relation between physical and mathematical objects is given by that of soap films and minimal surfaces \cite{Courant40}. If a soap film is spanned between two parallel metal rings with a small gap, its surface tension forces it to take the shape of a catenoid, which is a rotationally symmetric minimal surface. In contrast, for a large gap between the rings, such a surface does not exist, which manifests as a pinch-off (i.e.\ breaking) of the soap film bridge, see \cite{DK91,GPRS21}. Applying an additional force to the soap film may influence the occurrence of a pinch-off and also changes the shape of the film. Since soap responds to electrostatics \cite{MCD21}, the choice of an additional electrostatic force is possible, which is also particularly interesting due to its technological relevance \cite{Pelesko03}. A mathematical model for a soap film bridge subjected to an additional electrostatic force which counteracts the surface tension has been introduced in \cite{Moulton08,MP08, MP09} and consists of a singular ordinary differential equation (ODE).
In the present paper, we analytically study a stationary version of a new mathematical model \cite{LSS24a} describing the same physical set-up but derived under more general modelling assumptions. In particular, we deal with a free boundary problem instead of a singular ODE. In our investigation, we focus on the case where surface tension and electrostatics are nearly balanced and study how small changes of the electrostatic force affect the soap film bridge. 



 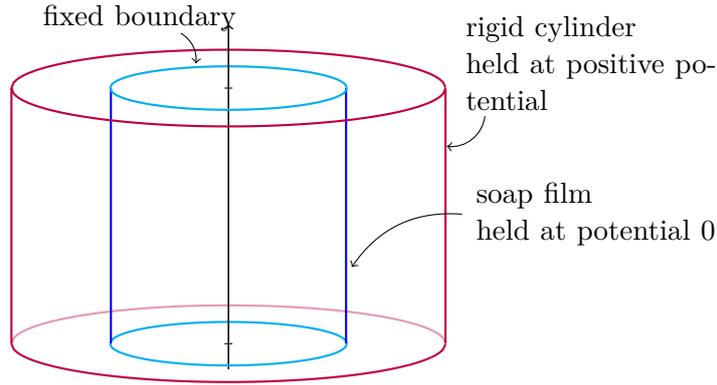
\begin{figure}[h]
\begin{tikzpicture}
\begin{axis}[samples=200, axis x line=none, axis y line=none, domain=-1:1, clip=false]
\addplot[color=cyan, thick]({0.6*x},{-1+0.17*(1-x^2)^0.5});
\addplot[color=cyan, thick]({0.6*x},{1+0.17*(1-x^2)^0.5});
\addplot[color=purple, thick]({1.1*x},{1+0.3*(1-x^2)^0.5});
\addplot[color=purple!40, thick]({1.1*x},{-1+0.3*(1-x^2)^0.5});
\addplot[->, domain=-1.2:1.5]({0},{x});
\addplot[domain=-0.02:0.02]({x},{-1});
\addplot[domain=-0.02:0.02]({x},{1});

\addplot[color=purple, thick]({-1.1},{x});
\addplot[color=purple, thick]({1.1},{x});
\addplot[->, domain=-1.2:1.5]({0},{x});

\addplot[color=blue, thick]({0.596},{x});
\addplot[color=blue, thick]({-0.596},{x});

\addplot[color=cyan,thick]({0.6*x},{(1-0.17*(1-x^2)^0.5)});
\addplot[color=cyan, thick]({0.6*x},{-1-0.17*(1-x^2)^0.5});

\addplot[color=purple,thick]({1.1*x},{(1-0.3*(1-x^2)^0.5)});
\addplot[color=purple, thick]({1.1*x},{-1-0.3*(1-x^2)^0.5});

\end{axis}

\draw [->, bend angle=45, bend left]  (6.8,4) to (6.3,3.6);
\node[text width=3.5cm] at (8.3,4.7) {\begin{small}rigid cylinder \\ held at positive potential \end{small}};

\draw [->, bend angle=30, bend right]  (6.5,2.7) to (5.05,2);
\node[text width=3.5cm] at (8.3,2.7) {\begin{small} soap film \\ \ held at potential $0$ \end{small}};

\draw [->, bend angle=30, bend left]  (2.8,5.1) to (3,4.7);
\node[text width=3.5cm] at (2.6,5.3) {\begin{small} fixed boundary \end{small}};
\end{tikzpicture}

\caption{Depiction of the soap film (dark blue), which is spanned between two parallel metal rings (light blue) and placed inside an outer metal cylinder (red). 
The electrostatic force is realized by applying a voltage. If the surface tension of the film and the electrostatic force are equally strong, the soap film bridge takes the shape of a cylinder.\label{fig1}
}
\end{figure}

\subsection{The Mathematical Model}
The precise set-up is given by a tiny soap film spanned between two metal rings and placed inside an outer metal cylinder, see Figure \ref{fig1}. Applying a voltage between this metal cylinder and the rings induces an electrostatic force pulling the film outwards. Throughout, we assume rotational symmetry, which allows us to consider the cross section, see Figure \ref{Cross Section}. More precisely, we describe the set-up in dimensionless form by a function $u=u(z): (-1,1) \rightarrow (-1,1)$ with $u(\pm1)=0$ such that $u+1$ gives the profile of the soap film bridge, and by the electrostatic potential $\psi=\psi(z,r): \overline{\Omega(u)} \rightarrow \mathbb{R}$ defined on the closure of the a-priori unknown domain between soap film and cylinder
$$\Omega(u) = \big \lbrace (z,r) \in (-1,1) \times (0,2) \, \vert \, u(z)+1 < r < 2 \, \big \rbrace.$$ 
In the stationary case, $u$ solves the elliptic equation
 \begin{align}
 \begin{cases}
  \qquad  -\sigma\, \partial_z \mathrm{arctan}  ( \sigma\partial_z u )&=  -\displaystyle\frac{1}{u+1} +\lambda\,(1+\sigma^2(\partial_z u)^2)^{3/2} \vert \partial_r \psi(z,u+1) \vert^2\,, \\
  \qquad \qquad \qquad u(\pm 1 )&=0 \,, \qquad -1 < u <1\,, 
 \end{cases} \label{eq1}
\end{align}
and the electrostatic potential $\psi$ is a solution to 
\begin{align}
\begin{cases}\displaystyle
\frac{1}{r} \partial_r \left ( r \partial_r \psi \right ) + \sigma^2 \partial^2_z \psi&= 0  \quad \ \text{in} \quad \Omega(u)\,,  \\
 \ \, \qquad \qquad \qquad \qquad \psi &= h_u \quad \text{on} \quad \partial \Omega(u) 
\end{cases} \label{eq2}
\end{align}
with
\begin{align}
h_u(z,r) = \frac{\ln \Big (\displaystyle\frac{r}{u(z)+1} \Big )}{\ln \Big (\displaystyle\frac{2}{u(z)+1} \Big )}\,. \label{eq3}
\end{align} 
The boundary condition \eqref{eq3}, which is $0$ on the film and $1$ on the cylinder, is due to a neglection of the fringing field as it is also common in other models for the interplay between surface tension and electrostatics \cite{ELW14, ELW15, LW17}. The parameter $\sigma$ gives the ratio of radii of the rings divided by their distance, a small value of $\sigma$ therefore indicating a larger distance between the rings, and $\lambda \in [0,\infty)$ measures the strength of the applied voltage. If $\lambda=0$, no voltage is applied, and \eqref{eq1} becomes a minimal surface equation which has a solution if and only if $\sigma \geq \sigma_{crit}$ with 
\begin{align}
\sigma_{crit} \approx 1.5\,,  \label{sigmacrit}
\end{align}
see \cite[p.\,282]{JLJ98}. We refer to \cite{LSS24} for the derivation of \eqref{eq1}-\eqref{eq3}. For later purposes, we also note that for the dynamical version of \eqref{eq1}-\eqref{eq3} in \cite{LSS24a} only the first line of \eqref{eq1} has to be replaced by
\begin{align}
\partial_t u - \sigma \partial_z \mathrm{arctan} (\sigma \partial_z u) = - \frac{1}{u+1} + \lambda \,(1+\sigma^2(\partial_z u)^2)^{3/2} \vert \partial_r \psi(z,u+1) \vert^2\,, \label{dyn}
\end{align} 
while in \eqref{eq2}-\eqref{eq3} time only occurs as a parameter. Moreover, an initial value $u_0$ is required. 

\subsection{Related Work}\label{CRW} 
First, as already mentioned, the model from \cite{Moulton08,MP08, MP09} describes the same physical set-up as the free boundary problem \eqref{eq1}-\eqref{eq3} herein but consists of a singular ODE which is derived by assuming a small aspect ratio for the radii difference between the rings and the cylinder divided by the distance of the rings, which is then formally put equal to zero. Besides modelling, in \cite{Moulton08}, the qualitative behaviour of solutions is studied by numerical and formal methods. Our reason for dropping the small ratio assumption is that there are other model parameters approximately of the same order as the aspect ratio. Second, we mention the model \cite{ELW15}. It belongs to a class of free boundary problems modelling small micro devices, which were introduced in \cite{LW13,ELW14}, see the survey article \cite{LW17}. The model in \cite{ELW15} is a free boundary problem for two unknowns $(u,\psi)$ consisting (in the stationary case) of an elliptic equation for $u$ and Laplace equation for $\psi$ in an a-priori unknown domain. Insofar, \eqref{eq1}-\eqref{eq3} has the same structure. However, the source term in the equation for $u$ in \cite{ELW15} differs significantly from that in \eqref{eq1} and reads
\begin{align*}
 - \lambda \,\big (1+(\sigma \partial_z u)^2 \big )\big \vert \partial_r \psi (\,x \,, u) \big \vert^2\,, 
\end{align*}
which has -- in contrast to the right hand side of \eqref{eq1} -- a fixed sign.  In \cite{ELW15}, among other things, existence and stability of stationary solutions is proven. However, the opposite signs in the source term of \eqref{eq1} yield a significantly different set of stationary solutions compared to \cite{ELW15}, and therefore proofs of qualitative properties of solutions to \eqref{eq1}-\eqref{eq3} require adaptation as well as new ingredients. Another question which arises is the direction in which $u$ deflects. A characterization of the direction of deflection for an evolution problem, in which the source term may consist of terms of opposite signs, is contained in \cite{EL17, L16} while the direction of deflection for stationary problems of $4$-th order, but with fixed sign of the source term, is treated in \cite{LW14a,Nik21b}. Finally, we mention that the present paper is supplemented by the author's work \cite{LSS24a, LSS24b}. In \cite{LSS24a} well-posedness and non-existence of global solutions for large $\lambda$ to the dynamical version of \eqref{eq1}-\eqref{eq3}, see \eqref{dyn}, is shown while \cite{LSS24b} focuses on \eqref{eq1}-\eqref{eq3} but for $\lambda$ taking values in a different parameter range.

\subsection{Cylinder as Stationary Solution} 
The starting point for our investigation is the stationary solution $u=0$, i.e.\ a cylinder, which occurs if surface tension and electrostatic force are perfectly balanced (at least if the fringing field is neglected). More precisely, for arbitrary $\sigma$, a direct computation shows that $u=0$ is a solution to \eqref{eq1} if $\lambda = \lambda_{cyl}$ with
\begin{align}
\lambda_{cyl} := \ln(2)^2\,. \label{eqCyl2}
\end{align}
The corresponding electrostatic potential solving \eqref{eq2}-\eqref{eq3} is then
\begin{align}
\psi(z,r)= \ln(r)/\ln(2)\,. \label{eqCyl1b}
\end{align}

\begin{figure}[h]
\begin{tikzpicture}
\fill[white,rounded corners] (-3.5,-0.8) rectangle (3.3,3.5);
\fill[purple!7] (-2.45,1.54) rectangle (2.45,3);
\draw[domain=-2.45:2.45, smooth, variable=\z, blue, thick] plot ({\z},{1.5});
\draw[purple, thick] (-2.45,3) -- (2.45,3) ;
\draw[dashed](2.45,0) -- (2.45,3) ;
\draw[dashed] (-2.45,3) -- (-2.45,0) ;
\draw[dashed] (-2.45,3) -- (-2.45,0) ;
\draw[dashed] (-2.45,3) -- (-2.45,0) ;
\draw[->] (-3,0) -- (3,0) node[right] {$z$};
\draw (2.45,0.1) -- (2.45,-0.1) node[below] {$1$};
\draw (-2.45,0.1) -- (-2.45,-0.1) node[below] {$-1$};
\draw[->] (-2.9,-0.1) -- (-2.9,3.2) node[right] {$r$};
\draw (-2.8,1.5) -- (-3,1.5) node[left] {$1$};
\draw (-2.8,3) -- (-3,3) node[left] {$2$};
\node[above] at (0,1.9) {$\Omega(u)$};

\draw[->, blue, thick] (-1.4,0) -- (-1.4,1.5);
\node[blue, right] at (-1.3,0.7) {$u+1$};
\draw [<-, bend angle=45, bend right]  (0.8,3) to (1,2.6) node[right]{$\psi=1$};
\draw [<-, bend angle=45, bend left]  (0.8,1.55) to (1,1.85) node[right]{$\psi=0$};

\end{tikzpicture}
\caption[Cross Section Soap Film Bridge in an Electric Field]{Cross section of the soap film bridge in an electric field. If the soap film bridge does not form a cylinder, then $u \neq 0$. }\label{Cross Section}
\end{figure}
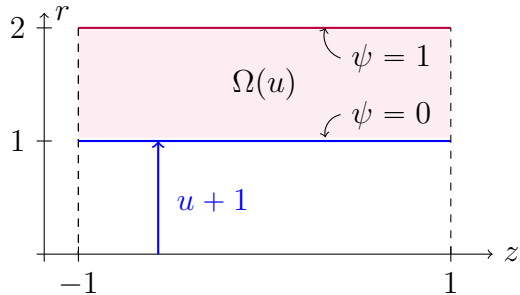\color{black}

\subsection{Main Results and Strategies}

We present results concerning existence, stability and direction of deflection for stationary solutions close to the cylinder in dependence on varying voltages $\lambda$. In all statements, the so-called eigencurve profile $[s \mapsto \mu(s)]$ will occur, which has a unique zero and can be written as
\begin{align*}
 \mu(s) = -s +3+ 2\, \partial_r h_s(1)\,, \qquad s \in (0,\infty)
\end{align*}
where $h_s \in W^2_{2,D}(-1,1)$ solves 
\begin{align*}
 \begin{cases}
&\begin{displaystyle} -\frac{1}{r} \partial_r \big ( r \,\partial_r h_s \big )  +s \, h_s = \frac{-2}{r^3} + s \, \frac{2-r}{r}\,, \end{displaystyle}\\
  &h_s(1)=h_s(2)=0\,,
 \end{cases}
\end{align*}
see Section \ref{SectionCylinder}. With $[s \mapsto \mu(s)]$ at hand, we can state the first result which focuses on existence of stationary solutions for $\lambda$ close to $\lambda_{cyl}$. 

\begin{thm}\label{exstatsLcyl}{\bf(Existence)}\\
Let $q \in (2,\infty)$, and $s_0 >0$ be the unique zero of the eigencurve profile $[s \mapsto \mu(s)]$.
Then, for each $\sigma>0$ with
$$\sigma^2 \neq \frac{4 \,s_0}{\pi^2 \,(j+1)^2}\,, \qquad j\in \mathbb{N}\,,$$
there exists $\delta =\delta(\sigma)> 0$ and an analytic function 
\begin{align*}
[\lambda \mapsto u^{\lambda}_{cyl}] &: (\lambda_{cyl}-\delta,\lambda_{cyl}+\delta) \rightarrow W^2_{q,D}(-1,1) \,, \ \qquad u_{cyl}^{\lambda_{cyl}}=0
\end{align*}
such that $u^\lambda_{cyl}$ is a solution to \eqref{eq1}-\eqref{eq3} for each $\lambda \in (\lambda_{cyl}-\delta,\lambda_{cyl}+\delta)$. Moreover, $u^\lambda_{cyl}$ as well as the corresponding electrostatic
potential $\psi_{u^\lambda_{cyl}} \in W^2_2\big (\Omega(u^\lambda_{cyl})\big)$ are symmetric with respect to the $r$-axis.
\end{thm}
 
Note that functions in $W^2_{q,D}(-1,1)$ satisfy Dirichlet boundary conditions. To prove Theorem \ref{exstatsLcyl}, we rewrite \eqref{eq1}-\eqref{eq3} as a non-local elliptic equation for $u$ only
 \begin{align}
 \begin{cases}
  &F(u)+ \lambda \, g(u)=0\,,\\
  &u(\pm1)=0 \,, \quad -1 <u < 1
 \end{cases}\label{statlambda1a}
 \end{align}
with non-local electrostatic force 
\begin{align}
  g(u):= (1+\sigma^2(\partial_z u)^2)^{3/2}\, \big \vert \partial_r \psi_u (z,u+1) \big \vert^2 \label{filmdimensionless0} 
\end{align}
 and
\begin{align}
 F(u):=\sigma \partial_z \mathrm{arctan} (\sigma \partial_z u)- \displaystyle\frac{1}{u+1}\,. \label{CatFa}
\end{align}
In \eqref{filmdimensionless0}, the notion $\psi_u$ highlights that we consider \eqref{eq2} in dependence on $u$. 
Theorem \ref{exstatsLcyl} follows from an application of the implicit function theorem which requires the linearization of problem \eqref{statlambda1a} around $(\lambda,u)=(\lambda_{cyl},0)$. The non-local character of $g$ makes the exact computation of the spectrum of the linearized operator $DF(0)+\lambda_{cyl} Dg(0)$ difficult. Instead, we derive qualitative properties of it using a  Fourier series ansatz \cite{AGM15, EM11b, EM11a, LS13}. For our specific problem, this yields to the investigation of the eigencurve profile $[s \mapsto \mu(s)]$.   \\
 
Second, we rigorously study stability of stationary solutions near the cylinder under rotationally symmetric perturbations. We find a sharp threshold value $\sigma_{cyl} > 0$ such that the stationary solution $u^\lambda_{cyl}$ from Theorem \ref{exstatsLcyl} is unstable for $\sigma< \sigma_{cyl}$ and stable for $\sigma>\sigma_{cyl}$. 
\color{black}

\begin{thm}\label{stability2}{\bf(Stability)}\\
 Let $q \in (2,\infty)$ and $\sigma^2 \neq \displaystyle\frac{4s_0}{\pi^2(j+1)^2}$ for $j \in \mathbb{N}$ and $s_0$ being the unique zero of the eigencurve profile $[s \mapsto \mu(s)]$. Define \begin{align}
  \sigma_{cyl}:= \sqrt{\frac{4s_0}{\pi^2}}\,, \label{thv}
 \end{align}
 Then, there exists $\delta >0$ such that for each $\lambda \in (\lambda_{cyl}-\delta,\lambda_{cyl}+\delta)$ the stationary solution $u_{cyl}^\lambda$ satisfies: \\
 {\bf(i)} If $\sigma < \sigma_{cyl}$, then $u_{cyl}^\lambda$ is unstable in $W^2_{q,D}(-1,1)$.\\
 {\bf(ii)} If $\sigma > \sigma_{cyl}$, then $u_{cyl}^\lambda$ is exponentially asymptotically stable in $W^2_{q,D}(-1,1)$. More precisely, 
 there exist numbers $\omega_0, m, M > 0$ such that for each initial value $u_0 \in W^2_{q,D}(-1,1)$ with $\Vert u_0 -u_{cyl}^\lambda \Vert_{W^2_{q,D}} < m$, the solution $u$ to the dynamical version of \eqref{eq1}-\eqref{eq3}, see \eqref{dyn}, exists globally in time and the estimate
 \begin{align*}
  \Vert u(t) -u_{cyl}^\lambda\Vert_{W^2_{q,D}(-1,1)} &+ \Vert \partial_t u(t) \Vert_{L_q(-1,1)} \leq M \, e^{-\omega_0 t} \Vert u_0  -u_{cyl}^\lambda \Vert_{W^2_{q,D}(-1,1)} 
 \end{align*}
 holds for $t \geq 0$.
\end{thm}

Here, we use the notion of unstable from \cite{Lunardi95}. The proof of Theorem \ref{stability2} is based on the principle of linearized stability. We point out that part (ii) contains an alternative argument for existence of solutions compared to \cite[Theorem 1.1]{LSS24a}, at least for $u_0$ close to $u^\lambda_{cyl}$ in the $W^2_q$-norm. Concerning the first two results, we note that the implicit function theorem and principle of linearized stability have been applied in \cite{ELW15}, too. However, as $\lambda_{cyl} > 0$ in our case, the implementation follows a different and more complicated route. \\

As a third result, for the stable range $\sigma > \sigma_{cyl}$, we show that the stationary solutions $u^\lambda_{cyl}$, stemming from the cylinder $u=0$, are deflected monotonically outwards with respect to $\lambda$: 

\begin{thm}\label{cyldirectionofdeflection}{\bf(Direction of Deflection)}\\
For $\sigma > \sigma_{cyl}$, there exists $\delta >0$ such that  
\begin{align*}
u^{\overline{\lambda}}_{cyl} (z) < u^\lambda_{cyl}(z) \,, \qquad  \lambda_{cyl}-\delta < \overline{\lambda} < \lambda < \lambda_{cyl}+\delta\,,\quad z \in (-1,1)\,.\\
\end{align*}
\end{thm}

This reflects the physically expected behaviour that increasing the electrostatic force, which is scaled by $\lambda$, pulls the film farther outwards. Theorem \ref{cyldirectionofdeflection} is in accordance with the behaviour of the simpler ODE model in \cite{Moulton08} where formal asymptotic analysis has been used to establish a similar result.\\

The proof of Theorem \ref{cyldirectionofdeflection} is based on the linear approximation
\begin{align}
u^\lambda_{cyl} &= u^{\lambda_{cyl}}_{cyl} + (\lambda-\lambda_{cyl})\,\partial_\lambda u^{\lambda_{cyl}}_{cyl} + o\big (\lambda-\lambda_{cyl} \big ) \nonumber \\
&=(\lambda-\lambda_{cyl})\, \partial_\lambda u^{\lambda_{cyl}}_{cyl} + o ( \lambda-\lambda_{cyl} )\,, \qquad \lambda \rightarrow \lambda_{cyl} \label{cylab0}
\end{align}
with
\begin{align}
\partial_\lambda u_{cyl}^{\lambda_{cyl}} &= - \big [DF(0) + \lambda_{cyl} Dg(0) \big]^{-1} g(0) \notag \\
&=-\,\frac{1}{\ln(2)^2}\, \big [DF(0) + \lambda_{cyl} Dg(0) \big]^{-1} \,\mathbbm{1} \label{cylabl}
\end{align}
in $W^2_{q,D}(-1,1)$. Here, we inserted $g(0)= \ln(2)^{-2}$, which easily follows from \eqref{filmdimensionless0}  combined with \eqref{eqCyl1b}, and used the fact that $[\lambda \mapsto u^\lambda_{cyl}]$
was constructed via the implicit function theorem. We now ask for positivity of \eqref{cylabl}, which is no easy question as the linearized operator in \eqref{cylabl} includes the non-local part $+\lambda_{cyl} Dg(0)$,
which most likely precludes the use of any standard method such as applying the maximum principle. Instead, we expand $\partial_{\lambda}u^{\lambda_{cyl}}_{cyl}$ in a Fourier series and show positivity of this series by hand.
Essential ingredients for the positivity proof are estimates on the eigencurve profile $[s \mapsto \mu(s)]$. \\

\subsection{Outline of the Paper}
We start with preliminaries in Section \ref{NP}. Then, in Section \ref{Sec2}, we present the linearization of \eqref{statlambda1a} around the cylinder $u=0$. We also derive first properties of its spectrum. Subsequently, in Section \ref{SectionCylinder}, we analyse the eigencurve profile $[s \mapsto \mu(s)]$, which enables us to prove the main results, Theorem \ref{exstatsLcyl} - Theorem \ref{cyldirectionofdeflection}, in Section \ref{Sec4}. Next, in Section \ref{Disc}, we discuss the relation between the characteristic parameters $\sigma_{crit}$ and $\sigma_{cyl}$, defined in \eqref{sigmacrit} and \eqref{thv}, and its interpretation as a balancing effect of the electrostatic force. We conclude with Appendix \ref{Appendices}, which contains some outsourced computations.

\section{Notations and Preliminaries}\label{NP}
Let $U \subset \mathbb{R}^n$ be open and bounded with a Lipschitz boundary. For $p \in (1,\infty)$ and $s\in (0,2]$ with $s \neq 1/p$, we put
\begin{align*}
 W^s_{p,D}(U):= \begin{cases} \ \ \ W^s_p(U) \quad \qquad \qquad \qquad \qquad \ \ \,\text{for} \quad s\in (0,1/p)\,, \\
  \ \ \big \lbrace f \in W^s_p(U) \, \big \vert \, f=0 \ \text{on} \ \partial U \, \big \rbrace \quad\,\text{for} \quad s\in (1/p,2]\,,
\end{cases} 
\end{align*}
where $W^s_p(U)$ denotes the usual fractional Sobolev space. For $A: W^2_{p,D}(-1,1) \rightarrow L_p(-1,1)$, we write $A \in \mathcal{H}  \big(W^2_{p,D}(-1,1), L_p(-1,1) \big )$
if $-A$ generates an analytic semigroup on $L_p(-1,1)$ with domain $W^2_{p,D}(-1,1)$, see \cite{AmannLQPP} for details on this theory.\\

If $E$ and $F$ are Banach spaces, we let $\mathcal{L}(E,F)$ be the Banach space of bounded linear operators from $E$ to $F$, and $\mathcal{L}_{is}(E,F)$ is the subspace of isomorphisms. Moreover, we write $E \hookrightarrow F$ if $E$ is continuously embedded in $F$, and $E \overset{c}{\hookrightarrow} F$ if the embedding is also compact.\\

For $\Omega=\big \lbrace (z,r)\in (-1,1)\times (1,2) \big \rbrace$, we let
\begin{align}
- \Delta_{cyl,D}: \ W^2_{2,D}(\Omega) \rightarrow L_2(\Omega) \,, \quad f \mapsto - \frac{1}{r} \partial_r \big ( r \partial_r f \big )  -\sigma^2 \partial_z^2 f\label{eqCyl7b}
\end{align}
be the Laplace operator in cylindrical coordinates. Moreover, $-\Delta_{cyl,D} \in \mathcal{L}_{is}(W^2_{2,D}(\Omega),L_2(\Omega))$ by \cite[Theorem 3.2.1.2]{Grisvard85}.\\

For the Fourier series ansatz, we will require the eigenvalues and eigenfunctions of $- \Delta_{cyl,D}$. Therefore, we introduce
$$L_{2,r}(1,2):= \Big (L_2(1,2), (\,\cdot \, \vert \, \cdot \, )_{L_{2,r}(1,2)} \Big )$$
with weighted scalar product 
$$(f \vert h )_{L_{2,r}(1,2)} := \int_1^2 f(r)\,h(r)\,r\,\mathrm{d} r \,, \qquad f,h \in L_2(1,2)\,,$$
and, analogously,
$$L_{2,r}(\Omega):= \Big (L_2(\Omega), (\,\cdot \, \vert \, \cdot \, )_{L_{2,r}(\Omega)} \Big )$$
with weighted scalar product 
$$(f \vert h )_{L_{2,r}(\Omega)} := \int_{-1}^1 \int_1^2 f(z,r)\,h(z,r)\,r\,\mathrm{d} r \,\mathrm{d} z \qquad f,h \in L_2(\Omega)\,.$$
These spaces are obviously isomorphic to $L_2(1,2)$ and $L_2(\Omega)$ respectively. We abbreviate both scalar products with $(\,\cdot \, \vert \, \cdot \,)_{L_{2,r}}$. \\

The operator $-\Delta_{cyl,D}$ splits into two parts: We recall that the spectrum of the one-dimensional Dirichlet-Laplacian $-\partial_z^2: W^2_{2,D}(-1,1) \rightarrow L_2(-1,1)$ consists entirely of eigenvalues
\begin{align*}
 \nu_j := \frac{(j+1)^2\pi^2}{4}   \,, \qquad j\in\mathbb{N}\,,
\end{align*}
with geometric multiplicity $1$ and corresponding normalized eigenfunctions
\begin{align*}
 \varphi_j(z):=\displaystyle\begin{cases} \,&\cos \left (\displaystyle\frac{(j+1)\pi}{2}\,z \right) \qquad \text{if $j$ is even}\,, \\ 
               \,&\sin \left (\displaystyle\frac{(j+1)\pi}{2}\,z \right) \qquad \text{if $j$ is odd}
              \end{cases} 
\end{align*}
for $ z \in (-1,1)$. The eigenfunctions $\lbrace \varphi_j\rbrace_{j \in\mathbb{N}}$ form an orthonormal basis of $L_2(-1,1)$. \\

The spectrum of $- \frac{1}{r}\partial_r (r \partial_r \, \cdot\,): W^2_{2,D}(1,2) \rightarrow L_{2}(1,2)$ consists entirely of eigenvalues 
\begin{align}
 0 < \xi_0 < \xi_1 < \dots < \xi_k \rightarrow \infty \label{ewxi}
\end{align}
with geometric multiplicity $1$. The corresponding sequence of normalized eigenfunctions $\lbrace \rho_k\rbrace_{k \in \mathbb{N}}$ belongs to $C^\infty \big ([1,2]\big)\cap W^2_{2,D}(1,2)$ and forms an orthonormal basis of $L_{2,r}(1,2)$.\\

Consequently, the spectrum of $-\Delta_{cyl,D}$ consists entirely of eigenvalues
\begin{align*}
\xi_k + \sigma^2 \nu_j \,, \quad j,k\in\mathbb{N}\,,
\end{align*}
and the corresponding eigenfunctions $\rho_k\varphi_j \in C^\infty (\overline{\Omega}) \cap W^2_{2,D}(\Omega)$ form an orthonormal basis of $L_{2,r}(\Omega)$. Moreover, for each $f \in W^2_{2,D}(\Omega)$, there exists $b_{jk}\in\mathbb{R}$ such that 
\begin{align*}
f= \sum_{j,k} b_{jk} \rho_k \varphi_j\,,
\end{align*}
where the sequence converges unconditionally in $W^2_2(\Omega)$.

\section{The Linearization}\label{Sec2}

In this section, we present the linearization $DF(0)+ \lambda_{cyl} Dg(0)$ for $\lambda_{cyl}= \ln(2)^2$ around the cylinder 
$u=0$. We also check that $DF(0)+\lambda_{cyl}Dg(0)$ generates an analytic semigroup and derive first properties of its spectrum.\\

We start by showing that the electrostatic force $g$ from \eqref{filmdimensionless0} is analytic, for which we follow \cite[Proposition 5]{ELW14}. Here, for $q \in (2,\infty)$, we study $g$ as a map from
\begin{align}
S:= \lbrace v \in W^2_{q,D}(-1,1) \, \vert \, -1 < v < 1 \rbrace \label{defS}
\end{align}
to $L_q(-1,1)$, where the choice of $S$ excludes film positions $v$ which pinch-off or touch the outer metal cylinder. Note that $g(v)$ contains the electrostatic potential $\psi_v$ whose dependency on $v$ is not obvious as $\psi_v$ solves Laplace equation on the $v$-dependent domain $\Omega(v)$. To resolve this issue, we rely on the same transformation to a fixed domain as in \cite{LSS24a}: For a given $v \in S$, we map the domain $\Omega(v)$ to the rectangle
 $$\Omega = (-1,1) \times \big ( 1,2 )$$
 through $T_v : \overline{\Omega(v)} \rightarrow \overline{\Omega}$ given by 
 \begin{align}
  T_v (z,r) := \left ( z, \frac{r-2v(z)}{1-v(z)} \right )  \,, \qquad (z,r)\in \overline{\Omega(v)}\,. \label{deftrafoT}
 \end{align}

Thanks to the chain rule and transformation results for Sobolev functions \cite[Lemma 2.3.2]{Necas12}, the electrostatic potential $\psi_v$ solves \eqref{eq3} strongly on $\Omega(v)$ if and only if the transformed electrostatic potential $\phi_v:= \psi_v \circ (T_v)^{-1}$ is a strong solution to 
 \begin{align}
 \begin{cases}
  -L_v \phi_v &=0 \qquad \quad \text{in} \quad \Omega\,,\\
	\ \ \ \,	\phi_v &= \displaystyle\frac{\mathrm{ln}(r)}{\mathrm{ln}(2)} \quad \text{on} \quad \partial \Omega\,.
  \end{cases}  \label{ellsub}
 \end{align}
 Here, the transformed $v$-dependent differential operator $-L_v$ can be written in divergence form in which it is again uniformly elliptic with $W^1_q$-coefficients depending analytically on $v$. For the precise form of $L_v$, we refer to Appendix \ref{LinearizationApp}. Putting
 \begin{align}
  L_D(v) \Phi := L_v \Phi\,, \qquad \Phi \in W^2_{2,D}(\Omega)\,, \quad v\in S \label{defLDv}
 \end{align}
which satisfies
$$L_D(v)\in  \mathcal{L}_{is}(W^2_{2,D}(\Omega) , L_2(\Omega))$$
thanks to \cite[Theorem 6.1]{LSS24a}, and letting $f_v:=L_v \frac{\ln(r)}{\ln(2)} \in L_2(\Omega)$, we find that 
\begin{align}
\phi_v := -L_D(v)^{-1} f_v +\frac{\ln(r)}{\ln(2)} \in W^2_2(\Omega)  \label{defphiv}
\end{align}
is the unique strong solution to the transformed electrostatic problem \eqref{ellsub}. 

Now, it is possible to express the electrostatic force
defined in \eqref{filmdimensionless0} in terms of the transformed electrostatic potential $\phi_v$ as follows
\begin{align}
 g(v) =   \big (1+\sigma^2 (\partial_z v )^2 \big)^{3/2} \, \frac{\vert \partial_r \phi_v(\, \cdot \,,1) \vert^2}{(1-v)^2} \,, \qquad v \in S\,. \label{g1}
\end{align}

In this formula, the dependency of $g$ on $v$ is accessible and we can prove the analogue to \cite[Proposition 5]{ELW14}:

\begin{prop}\label{propana}
Let $q \in (2,\infty)$. Then, the electrostatic force $g$ is analytic from $S$ to $L_q(-1,1)$.
\end{prop}

\begin{proof}
First, we note that the mappings
\begin{align*}
\left [  v \mapsto \frac{1}{1-v} \right ]\,,\quad \bigg [  v \mapsto \big (1+ \sigma^2 (\partial_z v)^2 \big )^{3/2} \bigg ]
\end{align*}
are analytic from $S$ to $W^1_{q}(-1,1)$, which follows from an adaptation of \cite[Example~4.3.6]{BT03} and from the fact that the composition of analytic maps is again analytic. Next, we deduce that the maps
\begin{align*}
[v \mapsto L_D(v)] : S \rightarrow \mathcal{L}\big (W^2_{2,D}(\Omega), L_2(\Omega)\big)\,, \qquad [v \mapsto f_v] : S \rightarrow L_2(\Omega)
\end{align*}
are also analytic so that the definition of $\phi_v$ from \eqref{defphiv} combined with the analyticity of the inversion map $[\ell \mapsto \ell^{-1}]$ for bounded linear operators implies that $[v \mapsto \phi_v]$ is analytic from $S$ to $W^2_2(\Omega)$ as well. Finally, the representation of $g$ in terms of $\phi_v$ in \eqref{g1} and the Multiplication Theorem for fractional Sobolev Spaces \cite[Theorem 4.1, Remark 4.2\,(d)]{Amann91} (see also \cite[Theorem 7.1]{ELW15}) yield the analyticity of $g$ from $S$ to $L_q(-1,1)$.
\end{proof}

Next, we present the linearization of \eqref{statlambda1a}  around $0$.  The lengthy computation is given in Appendix \ref{LinearizationApp}. 

\begin{lem}\label{Cyl3a}
The linearization of \eqref{statlambda1a} around $0$ is given by
\begin{align}
\big (DF(0)+\lambda_{cyl} Dg(0) \big ) v=\sigma^2 \partial_z^2 v +3v + 2 \, \partial_r  (-\Delta_{cyl,D}  )^{-1} \Big [- \frac{2}{r^3} v - \sigma^2\, \frac{2-r}{r}\, v_{zz} \Big ](\, \cdot \,,1) \,.  \label{lo}
\end{align}
for $v \in W^2_{q,D}(-1,1)$.
\end{lem}

\begin{proof}
By \eqref{DF(0)} and \eqref{Dg(0)} in Appendix \ref{LinearizationApp}, we have
\begin{align*}
DF(0)v &= \sigma^2 \partial_z^2 v +v\,, \\
\lambda_{cyl}\, Dg(0)v&= 2 v  + 2 \, \partial_r  (-\Delta_{cyl,D}  )^{-1} \Big [- \frac{2}{r^3} v - \sigma^2\, \frac{2-r}{r}\, v_{zz} \Big ](\, \cdot \,,1) \,, \qquad v \in W^2_{q,D}(-1,1)\,,
\end{align*}
from which the assertion follows.
\end{proof}

Now, we show that $DF(0)+\lambda_{cyl}Dg(0)$ is the generator of an analytic semigroup for which we rely on a perturbation result. Here, it is possible and important to include the case $q=2$ on which we comment in Remark \ref{GLFourier}.

\begin{prop}\label{Cyl1}
For $q \geq 2$, we have
\begin{align*}
- \big ( DF(0) + \lambda_{cyl} \,Dg(0) \big ) \in \mathcal{H} \big (W^2_{q,D}(-1,1),L_q(-1,1) \big ) \,.
\end{align*}
\end{prop}
\begin{proof}
First we note that
\begin{align*}
DF(0)v = \sigma^2 \partial_z^2 v +v
\end{align*}
is the generator of an analytic semigroup, i.e.
\begin{align}
 -DF(0)  \in \mathcal{H}\big (W^2_{q,D}(-1,1),L_q(-1,1)\big )\,. \label{eqCyl7}
\end{align}
Second, we note that the following composition of maps
\begin{align*}
L_2(\Omega) \overset{(-\Delta_{cyl,D})^{-1}}{\longrightarrow} W^{2}_{2,D}(\Omega) \overset{\partial_r}{\longrightarrow} W^{1}_{2}(\Omega) \overset{\mathrm{tr}}{\longrightarrow} W^{1/2}_2(-1,1) \overset{c}{\hookrightarrow} L_q(-1,1) 
\end{align*}
defines a compact linear operator from $L_2(\Omega)$ to $L_{q}(-1,1)$. The notion $\mathrm{tr}$ denotes the trace operator with respect to the boundary part $r \equiv 1$. 
Because the map $\big [ v \mapsto -2/r^3 \,v - \sigma^2 (2-r)/r \,v_{zz} \big ]$ is bounded from $W^2_{q,D}(-1,1)$ to $L_2(\Omega)$, it follows that 
$$\lambda_{cyl} Dg(0) \in \mathcal{L} \big (W^2_{q,D}(-1,1),L_q(-1,1) \big)$$
is compact as the composition of a compact and a bounded operator. Now, the assertion follows from \eqref{eqCyl7} and the perturbation result \cite[Proposition 2.4.3]{Lunardi95} (or \cite[Theorem I.1.5.1]{AmannLQPP}).
\end{proof}
\color{black}

\begin{bem}\label{GLFourier}
Considering the linearization $DF(0) + \lambda_{cyl} Dg(0)$ as a bounded linear operator (or even a generator of an analytic semigroup) from $W^2_{2,D}(-1,1)$ to $L_2(-1,1)$ is crucial for us as it allows us to work with the Fourier series. However, it is not possible to include the case $q=2$ from the beginning because the assumption $q>2$ is needed to even define the electrostatic force $g$ in \eqref{filmdimensionless0}. This is due to the corners of $\Omega(u)$ and we refer to \cite{LSS24a} for more explanations. \\
\end{bem}

 Based on this preparation, we can compute the Fourier representation of $DF(0)+\lambda_{cyl} Dg(0)$.
 
\begin{lem}\label{Cyl2}
For $q\geq2$ and $v \in W^2_{q,D}(-1,1)$, the linearized operator can be written as
\begin{align*}
(DF(0) + \lambda_{cyl} Dg(0) \big ) v =  \sigma^2 \partial_z^2 v + 3v +2 (B_1 +  B_2)v 
\end{align*}
in $L_2(-1,1)$ where
\begin{align*}
B_1 v :=  \sum_{j,k} \frac{c_k}{\xi_k + \sigma^2 \nu_j}\, \partial_r\rho_k(1)  \,(v \vert \varphi_j)_{L_2} \,\varphi_j 
\end{align*}
with (unconditional) convergence in $L_2(-1,1)$ and
\begin{align*}
 B_2 v:=\sum_{j,k}\frac{\sigma^2\nu_j\, d_k}{\xi_k + \sigma^2 \nu_j}\, \partial_r\rho_k(1) \,(v \vert \varphi_j)_{L_2} \,\varphi_j 
\end{align*}
with (unconditional) convergence in $L_2(-1,1)$. The coefficients $(c_k), \ (d_k) \in \ell_2$ are given by
\begin{align*}
c_k := \Big ( - \frac{2}{r^3} \, \Big \vert \, \rho_k \Big )_{L_{2,r}} \,, \qquad d_k := \Big ( \frac{2-r}{r} \, \Big \vert \, \rho_k \Big )_{L_{2,r}}\,,  \qquad k \in \mathbb{N} .
\end{align*}
\end{lem}

\begin{proof}
For $v \in W^2_{q,D}(-1,1) \hookrightarrow W^2_{2,D}(-1,1)$, we represent the corresponding solution $f \in W^2_{2,D} ( \Omega )$ to 
\begin{align}
(- \Delta_{cyl,D})f= - \frac{2}{r^3} v - \sigma^2 \frac{2-r}{r} v_{zz} \label{eqCyl12b}
\end{align}
 by its Fourier series
\begin{align*}
f = \sum_{j,k} b_{jk}\,\rho_k \, \varphi_j \,, \qquad  b_{jk} \in \mathbb{R}\,.
\end{align*}
From \eqref{eqCyl12b} and the orthogonality of $\lbrace \rho_k\,\varphi_j \rbrace_{j,k}$ in $L_{2,r}(\Omega)$, we deduce that
\begin{align*}
b_{jk} \big (\xi_k + \sigma^2 \nu_j \big) = 
&=(c_k + \sigma^2 \nu_j \,d_k ) \,(v \vert \varphi_j)_{L_2}\,.
\end{align*}
Hence, we have
\begin{align}
f = \sum_{j,k} \frac{c_k+\sigma^2 \nu_j\, d_k}{\xi_k + \sigma^2 \nu_j} \,(v \vert \varphi_j)_{L_2} \,\varphi_j\,\rho_k  \label{eqCyl12c}
\end{align}
and 
\begin{align}
\partial_r f( \, \cdot \,,1) = \sum_{j,k} \frac{c_k+\sigma^2 \nu_j\, d_k}{\xi_k + \sigma^2 \nu_j} \,\partial_r\rho_k(1)  \,(v \vert \varphi_j)_{L_2} \,\varphi_j  \label{eqCyl3}
\end{align}
in $L_2(-1,1$) with $(c_k)\,, \, (d_k) \in \ell_2$. Noting that
\begin{align*}
(DF(0) + \lambda_{cyl} Dg(0) \big ) v =  \sigma^2 \partial_z^2 v + 3v+ 2\, \partial_r f( \, \cdot \,,1)
\end{align*} 
thanks to \eqref{lo}, the assertion follows from \eqref{eqCyl3}.
\end{proof}

Next, we prove that all eigenvalues of the linearization are real and that they are given by scaled versions of the same profile function:

\begin{defin}\label{defprofil}
 We call $  \mu:  (0,\infty) \rightarrow \mathbb{R}$, given by 
 \begin{align}
\mu(s):= -s+3+2 \bigg [ \sum_k \frac{c_k}{\xi_k + s} \partial_r\rho_k(1) \bigg  ]+2s\,\bigg [ \sum_k \frac{d_k}{\xi_k +s} \partial_r\rho_k(1) \bigg  ]\,, \qquad s > 0\,, \label{nudef}
 \end{align}
{\it eigencurve profile} for $DF(0)+ \lambda_{cyl} Dg(0)$. The coefficients $(c_k)$ and $(d_k)$ are the same as in Lemma \ref{Cyl2}.
\end{defin}

The well-definedness of $[s \mapsto \mu(s)]$ is a consequence of the next Lemma \ref{Cyl3}, in which we establish a connection between $[s \mapsto \mu(s)]$ and the eigenvalues of the linearization $DF(0)+\lambda_{cyl}Dg(0)$. Furthermore, we point out that the final form of $[s \mapsto \mu(s)]$ will be derived in Lemma \ref{Cyl5}, which will also show that the eigencurve profile is defined in $s=0$.

\begin{lem}\label{Cyl3}
The spectrum of $DF(0)+ \lambda_{cyl} Dg(0)$ consists entirely of real eigenvalues with no finite accumulation point. These eigenvalues are given by
\begin{align*}
\mu_j(\sigma):=\mu \big (\sigma^2 \nu_j \big )
\end{align*}
for $j \in \mathbb{N}$ (at this stage possibly neither ordered nor distinct). An eigenfunction which coincides with the $j$-th eigenfunction $\varphi_j$ of the one-dimensional Dirichlet-Laplacian corresponds to each eigenvalue. 
\end{lem}

\begin{proof}
Because $W^2_{q,D}(-1,1)$ is compactly embedded in $L_q(-1,1)$, the spectrum of the complexification of the linearized operator consists only of eigenvalues with no finite accumulation point, see \cite[Theorem 6.29]{Kato95}. Moreover, Lemma \ref{Cyl2} ensures that 
$$\Big (\big (DF(0)+\lambda_{cyl} Dg(0)\big) w_1 \,\Big \vert \, w_2 \,\Big )_{L_2}= \Big (\, w_1 \,\Big \vert \, \big (DF(0)+\lambda_{cyl} Dg(0)\big )w_2 \,\Big )_{L_2} $$
for $w_1, w_2 \in W^2_{q,D}(-1,1)$ as well as
\begin{align}
\big (DF(0)+\lambda_{cyl} Dg(0)\big ) \varphi_j = \mu_j(\sigma) \varphi_j \,, \qquad j \in \mathbb{N}\,, 
\end{align}
for the $j$-th eigenfunction $\varphi_j$ of the one-dimensional Dirichlet-Laplacian. 
\end{proof}

\begin{section}{Qualitative Properties of the Eigencurve Profile}\label{SectionCylinder}

To further analyse the spectrum of the linearized operator $DF(0)+\lambda_{cyl} Dg(0)$, it suffices to investigate the eigencurve profile $[s \mapsto \mu(s)]$. In particular, we will show the following:

\begin{prop}\label{Cyl4}
The eigencurve profile $[s \mapsto \mu(s)]$ is strictly decreasing on $[0, \infty)$ and there exists $s_0 \in (0,\infty)$ with $\mu(s_0)=0$.
\end{prop}

The proof of Proposition \ref{Cyl4} is given after some preparation. As a first step towards it, we present another representation of the eigencurve profile $[s \mapsto \mu(s)]$, which includes the case $s=0$. We also compute the derivative of the eigencurve profile.
\begin{lem}\label{Cyl5}
The eigencurve profile $\mu$ may equivalently be written as 
\begin{align}
 \mu(s) = -s +3+ 2\, \partial_r h_s(1)\,, \qquad s \in (0,\infty)\label{nualt}
\end{align}
where $h_s \in W^2_{2,D}(-1,1)$ solves 
\begin{align}
 \begin{cases}
&\begin{displaystyle} -\frac{1}{r} \partial_r \big ( r \,\partial_r h_s \big )  +s \, h_s = \frac{-2}{r^3} + s \, \frac{2-r}{r}\,, \end{displaystyle}\\
  &h_s(1)=h_s(2)=0\,.
 \end{cases}\label{eqCyl13a}
\end{align}
This representation holds even for $s > - \xi_0$ with $\xi_0>0$ from \eqref{ewxi}. In particular, $\mu \in C^\infty \big ((-\xi_0,\infty), \mathbb{R} \big )$ with 
\begin{align}
\mu^\prime(s)= -1 + 2 \,\partial_r  p_s(1)\,, \label{eqCyl13b}
\end{align}
 where $p_s \in W^2_{2,D}(-1,1)$ solves
\begin{align}
 \begin{cases}
  &\begin{displaystyle} -\frac{1}{r} \partial_r \big ( r \,\partial_r p_s \big )  +s\, p_s= \frac{2-r}{r} - h_s  \,, \end{displaystyle}\\
  &p_s(1)=p_s(2)=0\,.
 \end{cases} \label{eqCyl13c}
\end{align}
\end{lem}

\begin{proof}
{\bf(i)} We derive the alternative formula \eqref{nualt} for $\mu$: For $s \in (0,\infty)$, we find $\sigma \in (0,\infty)$ such that $s = \sigma^2 \nu_0$. Let us note that the solution $f_s$ to 
\begin{align*}
(-\Delta_{cyl,D})f_s &= - \frac{2}{r^3} \varphi_0 - \sigma^2 \frac{2-r}{r} \partial_z^2\varphi_0 \nonumber \\
&= \Big (- \frac{2}{r^3} + s\, \frac{2-r}{r} \Big )\, \varphi_0 
\end{align*}
with $\varphi_0$ denoting the first eigenfunction of the one-dimensional Dirichlet-Laplacian can be written in the form 
\begin{align}
f_s(z,r)= h_s(r)\,\varphi_0(z) \label{eqCyl13}
\end{align}
with $h_s \in C^\infty([1,2])$ solving
\begin{align*}
\begin{cases}
&\begin{displaystyle}-\frac{1}{r} \partial_r \big (r\, \partial_r h_s \big ) + s \, h_s = - \frac{2}{r^3} + s\, \frac{2-r}{r} \,, \quad r \in (1,2)\,, \end{displaystyle} \\
&h_s(1)=h_s(2)=0\,.
\end{cases}
\end{align*}
We now derive from the relation $s = \sigma^2 \nu_0$, \eqref{lo} and \eqref{eqCyl13} that 
\begin{align*}
 \big (DF(0)+ \lambda_{cyl} Dg(0) \big ) \varphi_0 &=      (-s+3) \varphi_0 + 2\, \partial_r f_s(\, \cdot \,,1)
 \\&=\big  (-s + 3 + 2\, \partial_r h_s(1) \big ) \varphi_0\,.
\end{align*}
Combining this with Lemma \ref{Cyl3} yields
$$\mu(s) =-s + 3 + 2\, \partial_r h_s(1) \,, \qquad s \in (0,\infty)\,,$$
which is formula \eqref{nualt}.\\
{\bf(ii)} Note that the operator $-\frac{1}{r} \partial_r \big ( r \partial_r \, \cdot \,) + s$ is invertible for each $s \in (-\xi_0, \infty)$.
Because the right-hand side of \eqref{eqCyl13a} depends smoothly on $s$, and taking the inverse is a smooth operation, it follows that $\mu \in C^\infty \big ( (-\xi_0,\infty), \mathbb{R})$. Moreover, its derivative is given by 
\begin{align*}
 \mu^\prime(s)= -1 +2 \,\partial_r p_s (1)\,,
\end{align*}
where $p_s:= \partial_s h_s \in W^2_{2,D}(-1,1)$. Finally, we note that taking the derivative of both sides of \eqref{eqCyl13a} with respect to $s$ results in \eqref{eqCyl13c}. This shows the remaining formula \eqref{eqCyl13b} for $\mu^\prime$.
\end{proof}

\begin{bem}
 Note that the solution $h_s$ to \eqref{eqCyl13a} can be expressed in terms of Bessel functions of the first and second kind. Nevertheless, this expression is lengthy, and we were not able to deduce properties of the eigencurve profile from it. \\
\end{bem}

For the special case $s=0$, it is possible to give explicit formulas for $\mu(0)$ and $\mu^\prime(0)$: 

\begin{lem}\label{nu0andnu0prime}
The values $\mu(0)$ and $\mu^\prime(0)$ are given by
 \begin{align*}
 \mu(0)= -1+ \frac{2}{\ln(2)} >0
 \end{align*}
and
\begin{align*}
 \mu^\prime(0)=-2+ \frac{3}{2\ln(2)^2} - \frac{1}{\ln(2)} < -\frac{3}{10}\,.
\end{align*}

\end{lem}

\begin{proof}
 {\bf(i)} For $\mu(0)$, we note that \eqref{eqCyl13a} with $s=0$ reads
\begin{align*}
 \begin{cases}
&\begin{displaystyle} -\frac{1}{r} \partial_r \big ( r \,\partial_r h_0 \big )   = \frac{-2}{r^3}\,,  \end{displaystyle}\\
  &h_0(1)=h_0(2)=0\,.
 \end{cases}
\end{align*}
 This equation is solved by
 \begin{align}
 h_0(r)= \frac{2-r}{r} + \frac{\ln(r)}{\ln(2)} -1 \label{eqCyl14a}
 \end{align}
with derivative
 \begin{align}
 \partial_r h_0(1)= -2+ \frac{1}{\ln(2)}\,. \label{eqCyl14aa}
 \end{align}
Hence, equation \eqref{nualt} yields
 \begin{align*}
 \mu(0)= -1+ \frac{2}{\ln(2)} > 0 \,.
 \end{align*}
{\bf(ii)} For $\mu^\prime(0)$, we first recall from \eqref{eqCyl13b} that
\begin{align*}
\mu^\prime(0)= -1+ 2\,\partial_r p_{0}(1)\,.
\end{align*}
 The function $p_0$ solves 
\begin{align*}
 \begin{cases}
  &\begin{displaystyle} -\frac{1}{r} \partial_r \big ( r \,\partial_r p_0 \big )  = 1- \frac{\ln(r)}{\ln(2)}\,, \end{displaystyle}\\
  &p_0(1)=p_0(2)=0\,,
 \end{cases}
\end{align*} 
which is \eqref{eqCyl13c} with $s=0$ and the inserted expression for $h_0$ from \eqref{eqCyl14a}. This equation has the explicit solution
\begin{align*}
p_0(r) =  \displaystyle \bigg ( \frac{3-\ln(2)}{4\ln(2)^2} \bigg ) \ln(r) +   \frac{1+\ln(2)}{4\ln(2)}   + \frac{r^2 \,\ln(r/2)-r^2}{4\ln(2)}
\end{align*}
with 
\begin{align*}
\partial_r p_0(1) &= \displaystyle \frac{3-\ln(2)}{4\ln(2)^2} + \frac{-2\ln(2) +1-2}{4\ln(2)} \\
&= \displaystyle \frac{3}{4 \ln(2)^2}- \frac{1}{2\ln(2)} - \frac{1}{2} \,.
\end{align*}
Plugging $\partial_r p_0(1)$ into the formula for $\mu^\prime(0)$ yields the assertion.
\end{proof}

Based on this preparation, we provide a proof that $[s \mapsto \mu(s)]$ is strictly decreasing on $[0,\infty)$ and has exactly one zero.\\

{\bf Proof of Proposition \ref{Cyl4}.}
 As $\mu$ is smooth with $\mu(0)>0$ as well as $\mu^\prime(0)<0$ by Lemma \ref{Cyl5} and Lemma \ref{nu0andnu0prime}, it is enough to show that $[s \mapsto \mu^\prime(s) ]$ is decreasing for $s \geq 0$. We will achieve that by applying the weak maximum principle several times.\\
{\bf(i)} First, we apply it to $[s \mapsto h_s]$, where we recall from \eqref{eqCyl13a} that $h_s$ solves 
\begin{align*}
 \begin{cases}
&\begin{displaystyle} -\frac{1}{r} \partial_r \big ( r \,\partial_r h_s \big )  +s \, h_s = \frac{-2}{r^3} + s \, \frac{2-r}{r}\,, \end{displaystyle}\\
  &h_s(1)=h_s(2)=0\,.
 \end{cases}
\end{align*}
Because $\frac{2-r}{r}$ solves 
\begin{align*}
\begin{cases}
&\displaystyle-\frac{1}{r} \partial_r \big (r\, \partial_r f(r) \big ) =- \frac{2}{r^3} \,, \\
&f(1)= 1 \,, \quad f(2) = 0 \,,
\end{cases}
\end{align*}
it follows that $\frac{2-r}{r}-h_s$ is a solution to 
\begin{align*}
\begin{cases}
&-\displaystyle\frac{1}{r} \partial_r \big (r\, \partial_rf(r) \big ) + s \, f(r)= 0 \,,\\
&f(1)=1 \,, \quad f(2)=0 \,.
\end{cases}
\end{align*}
An application of the weak maximum principle yields 
\begin{align}
h_s  \leq  \frac{2-r}{r}\,, \qquad s \geq 0\,. \label{eqCyl14}
\end{align} 
For $s>\tilde{s} \geq 0$, the difference $h_{s} - h_{\tilde{s}}$ solves 
\begin{align*}
\begin{cases}
&\displaystyle - \frac{1}{r} \partial_r \big (r \partial_r f(r)\big ) +  s f(r)   =  ( s-\tilde{s})\Big( \frac{2-r}{r} -h_{\tilde{s}}(r) \Big) \,, \\
&f(1)=f(2)=0\,,
\end{cases}
\end{align*}
where the right-hand side is non-negative thanks to \eqref{eqCyl14}. Consequently, the weak maximum principle yields
\begin{align}
h_s  \geq h_{\tilde{s}} \,, \qquad s > \tilde{s} \geq 0\,. \label{eqCyl15}
\end{align}
{\bf(ii)} Now, we apply the weak maximum principle for $s \geq 0$ to the solution $p_s=\partial_s h_s$ to \eqref{eqCyl13c}, i.e. to
\begin{align*}
\begin{cases}
  &\begin{displaystyle} -\frac{1}{r} \partial_r \big ( r \,\partial_r p_s \big )  +s\, p_s= \frac{2-r}{r} - h_s  \,, \end{displaystyle}\\
  &p_s(1)=p_s(2)=0\,.
 \end{cases}
\end{align*}
Due to \eqref{eqCyl14}, the right-hand side of this equation is non-negative and the weak maximum principle yields $p_s \geq 0$. For $s > \tilde{s} \geq 0$, we then find that $p_s - p_{\tilde{s}}$ solves
\begin{align*}
\begin{cases}
  &\begin{displaystyle} -\frac{1}{r} \partial_r \big ( r \,\partial_r f \big )  +s\, f= (h_{\tilde{s}} - h_s )+(\tilde{s}-s)p_{\tilde{s}}\,,  \end{displaystyle}\\
  &f(1)=f(2)=0
 \end{cases}
\end{align*}
with the non-positive right-hand side thanks to $p_{\tilde{s}} \geq 0$ and \eqref{eqCyl15}. Applying the weak maxiumum principle once more, we see that $p_s-p_{\tilde{s}}$ attains its maximum at $r=1$ and hence 
\begin{align*}
\partial_r p_s (1) \leq \partial_r p_{\tilde{s}}(1) \,, \qquad s > \tilde{s} \geq 0\,,
\end{align*}
which shows $\mu^\prime(s) \leq \mu^\prime(\tilde{s})$ for $s > \tilde{s} \geq 0$ as claimed.
\qed\\

\begin{bem}\label{cylincreasing} Since $h_s(1)=h_{\tilde{s}}(1)=0$ and $h_s - h_{\tilde{s}} \geq 0$ on $[1,2]$ for $s > \tilde{s} \geq 0$ by   \eqref{eqCyl15}, it follows that $[s \mapsto \partial_r h_s(1)]$ is increasing. 
\end{bem}

Through the relation
$$\mu_j(\sigma) = \mu(\sigma^2 \nu_j )\,, \qquad j \in \mathbb{N}\,,$$
where $\mu_j(\sigma)$ are the eigenvalues of the linearized operator $DF(0)+\lambda_{cyl}Dg(0)$, we can derive properties of its spectrum from properties of its eigencurve profile:
 
\begin{lem}\label{Cyl3aa}
The eigenvalues of $DF(0)+\lambda_{cyl}Dg(0)$ are ordered 
\begin{align*}
\mu_0(\sigma) > \mu_1(\sigma) > \dots > \mu_j(\sigma) > \mu_{j+1}(\sigma)> \dots \,, \qquad j \in \mathbb{N}\,,
\end{align*}
and have geometric multiplicity $1$.
\end{lem}

\begin{proof}
Because $[s \mapsto \mu(s)]$ is strictly decreasing by Proposition \ref{Cyl4} and the eigenvalues of the one-dimensional Dirichlet-Laplacian $(\nu_j)$ are strictly increasing, the eigenvalues $(\mu_j(\sigma))$ are strictly decreasing.
\end{proof}
 
Further properties of the eigenvalues $\mu_j(\sigma)$, following from properties of the eigencurve profile $[s\mapsto\mu(s)]$, are derived within the proofs in the next section:
\end{section}

\color{black}

\begin{section}{Proofs of the Main Results}\label{Sec4}

\begin{subsection}{Existence}\label{SectionECyl}
 We establish existence of stationary solutions for $\lambda$ close to $\lambda_{cyl}$:\\
 
{\bf Proof of Theorem \ref{exstatsLcyl}.}
Recall that
 \begin{align*}
S=\big \lbrace w \in W^2_{q,D}(-1,1) \ \big \vert \, -1 < w < 1\big \rbrace \,.
 \end{align*}
In the following, we want to resolve equation \eqref{statlambda1a}, that is $F(w) + \lambda g(w)=0$  
with $F$ from \eqref{CatFa}, locally around $(w, \lambda)=(0,\lambda_{cyl})$. Because $F$ and $g$ (see Proposition \ref{propana}) are analytic from $S$ to $L_q(-1,1)$ and the spectrum of 
$DF(0)+\lambda_{cyl} Dg(0)$ consists only of eigenvalues,
this is possible if and only if $0$ is no eigenvalue of $DF(0)+\lambda_{cyl} Dg(0)$. For $j\in \mathbb{N}$, we have 
\begin{align*}
 \sigma^2 \neq \frac{4 \,s_0}{\pi^2 \,(j+1)^2} \qquad \Longleftrightarrow \qquad \sigma^2 \nu_j \neq s_0  \qquad \Longleftrightarrow \qquad \mu_j(\sigma)=\mu(\sigma^2 \nu_j) \neq 0,
\end{align*}
and the implicit function theorem (in the form \cite[Theorem 4.5.4]{BT03}) is applicable. It yields some $\delta >0$ and an analytic function 
\begin{align*}
[\lambda \mapsto u^{\lambda}_{cyl}] &: (\lambda_{cyl}-\delta,\lambda_{cyl}+\delta) \rightarrow W^2_{q,D}(-1,1) \,, \ \qquad u_{cyl}^{\lambda_{cyl}}=0
\end{align*}
such that $u^\lambda_{cyl}$ is a solution to \eqref{statlambda1a} for each $\lambda \in (\lambda_{cyl}-\delta,\lambda_{cyl}+\delta)$ with
 \begin{align*}
\Vert u^\lambda_{cyl} \Vert_{W^2_{q,D}(-1,1)} < \delta\,.
\end{align*}
Additionally, if $u$ solves \eqref{statlambda1a} for some $\lambda \in (\lambda_{cyl}-\delta,\lambda_{cyl}+\delta)$ with
\begin{align}
\Vert u  \Vert_{W^2_{q,D}(-1,1)}  < \delta\,, \label{impunicyl}
\end{align}
then $u = u^\lambda_{cyl}$. From the uniqueness of the electrostatic potential in \eqref{eq2}, we deduce that $\big [z \mapsto u^\lambda_{cyl}(-z) \big ]$ is a second solution to \eqref{statlambda1a} having the same $W^2_q$-distance to $0$ as $u^\lambda_{cyl}$. Hence, it follows from \eqref{impunicyl} that $u^\lambda_{cyl}(-z)=u^\lambda_{cyl}(z)$. As a consequence, the electrostatic potential $\psi_{u^\lambda_{cyl}}$ is symmetric with respect to the $r$-axis.
\qed

\end{subsection}

\begin{subsection}{Stability}\label{SectionSCyl}

We want to apply the (rigorous) principle of linearized stability to establish Theorem \ref{stability2}. We roughly follow \cite{AGM15, EM11b, EM11a, LS13}.\\

For a solution $u \in W^2_{q,D}(-1,1)$ to the dynamical version of \eqref{eq1}-\eqref{eq3}, see \eqref{dyn}, with initial value $u_0$ close to $u^\lambda_{cyl}$, we put $v:=u-u_{cyl}^\lambda$. Then, $v$ satisfies the linearized equation
\begin{align}
\partial_t v- \big (DF(u^\lambda_{cyl})+\lambda Dg(u^\lambda_{cyl}) \big ) v &= F(u^\lambda_{cyl}+v)-F(u^\lambda_{cyl})-DF(u^\lambda_{cyl})v \nonumber \\
&+\lambda \big ( g(u^\lambda_{cyl}+v) -g(u^\lambda_{cyl}) -Dg(u^\lambda_{cyl})v \big )=:G_{cyl}(v). \label{eqCyl6}
\end{align}
Thanks to Proposition \ref{propana}, we have $G_{cyl} \in C^\infty \big (\mathcal{O}, L_q(-1,1) \big )$ for a small neighbourhood $\mathcal{O}$ of $0$ in $W^2_{q,D}(-1,1)$ satisfying $G_{cyl}(0)=0$ as well as $DG_{cyl}(0)=0$. \\

First, we study the stability of the cylinder $u^{\lambda_{cyl}}_{cyl}=0$:

\begin{lem}\label{stability2a}
Let $q \in (2,\infty)$ and $\lambda=\lambda_{cyl}$. Then, the following holds: \\
 {\bf(i)} If $\sigma < \sigma_{cyl}$, then the stationary solution $u=0$ to \eqref{eq1}-\eqref{eq3} is unstable in $W^2_{q,D}(-1,1)$.\\
 {\bf(ii)} If $\sigma > \sigma_{cyl}$, then the stationary solution $u=0$ to \eqref{eq1}-\eqref{eq3} is exponentially asymptotically stable in $W^2_{q,D}(-1,1)$.
\end{lem}
\begin{proof}
 Because of \eqref{eqCyl6} and $-\big(DF(0)+\lambda_{cyl}Dg(0) \big)\in\mathcal{H} \big (W^2_{q,D}(-1,1),L_q(-1,1) \big)$ by Proposition \ref{Cyl1}, we can apply the results from \cite{Lunardi95}. The choice of $\sigma_{cyl}$ in \eqref{thv} guarantees that the largest eigenvalue $\mu_0(\sigma)$ of $DF(0)+ \lambda_{cyl} Dg(0)$ satisfies 
\begin{align*}
\mu_0(\sigma)=\mu(\sigma^2 \nu_0)\ \begin{cases} 
&< 0\,, \qquad \sigma > \sigma_{cyl} \,, \\
&> 0 \,, \qquad \sigma < \sigma_{cyl}\,.
\end{cases}
\end{align*}
Hence, the assertion follows from \cite[Theorem 9.1.2, Theorem 9.1.3]{Lunardi95}. 
\end{proof} 

Second, we transfer the (in-)stability of the cylinder to the stationary solutions $u^\lambda_{cyl}$ going through the cylinder. To transfer the instability result, we require that $\mu_0(\sigma)$ is algebraically simple:

\begin{lem}\label{stability2aa}
The eigenvalue $\mu_0(\sigma)$ of $DF(0)+\lambda_{cyl} Dg(0)$ is algebraically simple in the sense of \cite[Definition A.2.7]{Lunardi95}.
\end{lem}

\begin{proof}
For simplicity, we put $A:= DF(0)+ \lambda_{cyl} Dg(0)$ and write $\mu_j$ for the $j$-th eigenvalue $\mu_j(\sigma)$ of $A$ throughout this proof. Since $\mu_0$ has geometric multiplicity $1$ due to Lemma \ref{Cyl3aa}, it remains to check that $\mu_0$ is semi-simple. Because $W^2_{q,D}(-1,1)$ is compactly embedded in $L_q(-1,1)$, the operator $A$ has a compact resolvent and \cite[1.19 Corollary]{EN00} ensures that $\mu_0$ is a pole of the resolvent of $A$. Consequently, \cite[Remark A.2.4]{Lunardi95} shows that $\mu_0$ is semi-simple if and only if
\begin{align}\mathrm{ker}  ( \mu_0 - A)^2 = \mathrm{ker}  (\mu_0 -A) = \mathbb{R} \cdot \varphi_0  \,,\label{semisimple}
\end{align}
where $\varphi_0$ is the first eigenvalue of the one-dimensional Dirichlet-Laplacian, see Lemma \ref{Cyl3}. Here, the equality \eqref{semisimple} follows from a direct computation using Fourier series.
\end{proof}

{\bf Proof of Theorem \ref{stability2}.}
For $\sigma \in (0,\infty)$ with $\sigma^2 \neq \displaystyle\frac{4 s_0}{\pi^2(j+1)^2}$ for $j\in \mathbb{N}$, we have
\begin{align*}
\Vert DF(u_{cyl}^\lambda) &+\lambda Dg(u_{cyl}^\lambda) - DF(0)-\lambda_{cyl} Dg(0) \Vert_{\mathcal{L}(W^2_{q,D},L_q)} \\
&\leq \Vert DF(u_{cyl}^\lambda) - DF(0) \Vert_{\mathcal{L}(W^2_{q,D},L_q)} +\lambda \Vert Dg(u^\lambda_{cyl})-Dg(0) \Vert_{\mathcal{L}(W^2_{q,D},L_q)}\\
&\ \ \ +\vert \lambda-\lambda_{cyl} \vert\, \Vert  Dg(0) \Vert_{\mathcal{L}(W^2_{q,D},L_q)} \rightarrow 0 \,,
\end{align*}
as $\lambda \rightarrow \lambda_{cyl}$ by Theorem \ref{exstatsLcyl}. Since the linearized operator $-\big(DF(0)+\lambda_{cyl} Dg(0)\big)$ belongs to $\mathcal{H}\big (W^2_{q,D}(-1,1),L_q(-1,1) \big)$, we deduce from \cite[Theorem I.1.3.1\,(i)]{AmannLQPP} the existence of $\delta > 0$ such that  
$$-\big (DF(u_{cyl}^\lambda) +\lambda Dg(u_{cyl}^\lambda)\big )\in \mathcal{H}\big (W^2_{q,D}(-1,1),L_q(-1,1) \big)\,, \qquad \lambda \in  (\lambda_{cyl}-\delta,\lambda_{cyl}+\delta)\,.$$
We now investigate the stability of $u^{\lambda_{cyl}}_{cyl}$ for $\sigma < \sigma_{cyl}$ and $\sigma > \sigma_{cyl}$ separately:\\
 
{\bf(i)} {\it Instability for $\sigma < \sigma_{cyl}$}: In this case, we know that the first eigenvalue $\mu_0(\sigma)$ of the operator $DF(0)+\lambda_{cyl}Dg(0)$ is positive. 
Because it is also isolated and algebraically simple by Lemma \ref{stability2aa}, the perturbation result \cite[Proposition A.3.2]{Lunardi95} for such eigenvalues allows us to make $\delta > 0$ smaller such that $DF(u_{cyl}^\lambda) +\lambda Dg(u_{cyl}^\lambda)$
also has an eigenvalue with positive real part for $\lambda \in (\lambda_{cyl}-\delta,\lambda_{cyl}+\delta)$. Moreover, 
since the embedding $W^2_{q,D}(-1,1) \hookrightarrow L_q(-1,1)$ is compact, the spectrum of $DF(u_{cyl}^\lambda) +\lambda Dg(u_{cyl}^\lambda)$
consists only of eigenvalues with no finite accumulation point, see \cite[Theorem 6.29]{Kato95}. Thus, there is a constant $C>0$ such that the strip $\big \lbrace \mu \in \mathbb{C} \, \big \vert \, 0 < \mathrm{Re} \mu \, < C \big \rbrace$
is contained in the resolvent set of $DF(u^\lambda_{cyl})+\lambda Dg(u^\lambda_{cyl})$. Applying now \cite[Theorem 9.1.3]{Lunardi95} to \eqref{eqCyl6} shows the instability of $u^\lambda_{cyl}$ for $\sigma < \sigma_{cyl}$. \\

{\bf(ii)} {\it Stability for $\sigma > \sigma_{cyl}$}: Since the spectral bound of $DF(0)+\lambda_{cyl} Dg(0)$ is negative due to the choice $\sigma > \sigma_{cyl}$, it follows from \cite[Corollary I.1.4.3]{AmannLQPP} that we may take $\delta >0$ so small that also $DF(u_{cyl}^\lambda) +\lambda Dg(u_{cyl}^\lambda)$ 
has a negative spectral bound for $\lambda \in (\lambda_{cyl}-\delta,\lambda_{cyl}+\delta)$. Hence, $u^\lambda_{cyl}$ is exponentially asymptotically stable by \cite[Theorem 9.1.2]{Lunardi95}.\\
\qed

\end{subsection}

\begin{subsection}{Direction of Deflection}\label{SectionDoDCyl}

Finally, we turn to Theorem \ref{cyldirectionofdeflection}, which states that the local branch of stationary solutions $[\lambda \mapsto u^\lambda_{cyl}]$ going through the stable cylinder is deflected monotonically outwards if the applied voltage is increased. \\

We start with the Fourier series of the function $\mathbbm{1}:=[z \mapsto 1]$:

\begin{lem}\label{consteig}
The Fourier series of $\mathbbm{1}$ with respect to the $L_2$-eigenbasis of the one-dimensional Dirichlet-Laplacian is 
\begin{align*}
1= \frac{4}{\pi} \sum_{j=0}^\infty \frac{(-1)^j}{(2j+1)} \cos \left(\frac{(2j+1)\pi}{2} z \right ) \,, \qquad z \in (-1,1)\,,
\end{align*}
with (unconditional) convergence in $L_2(-1,1)$.
\end{lem}
\begin{proof} 

This follows from
\begin{align*}
(\mathbbm{1} \vert \varphi_{2j})_{L_2}&= \int_{-1}^1 \cos \left( \frac{(2j+1) \pi}{2}  \,z \right) \, \mathrm{d} z \\
&= \frac{4}{(2j+1)\pi} \sin \left ( \frac{(2j+1) \pi}{2}  \right) = \frac{4}{\pi} \frac{(-1)^j}{(2j+1)}\,, \qquad j \in \mathbb{N}\,,
\end{align*}
and $(\mathbbm{1} \vert \varphi_{2j+1})_{L_2}=0$ since $\mathbbm{1}$ is even.
\end{proof}

Next, we compute the Fourier series of $\partial_{\lambda}u_{cyl}^{\lambda_{cyl}}$:

\begin{lem}\label{directioncyl3a}
Let $\sigma > \sigma_{cyl}$. Then, the Fourier series of $\partial_\lambda u_{cyl}^{\lambda_{cyl}}$ is
\begin{align}\partial_\lambda u^{\lambda_{cyl}}_{cyl}(z) =\frac{4}{\pi\ln(2)^2} \sum_{j=0}^\infty a_j \cos \bigg ( \frac{(2j+1)\pi}{2} z \bigg )\label{fseries} \end{align}
with coefficients
\begin{align*}
a_j:=\frac{(-1)^j}{(2j+1)\big(-\mu_{2j}(\sigma)\big)}\,, \qquad j \in \mathbb{N}\,,
\end{align*}
and (unconditional) convergence in $C^1\big([-1,1]\big)$. Here, $\mu_{2j}(\sigma)$ denotes the $2j$-th eigenvalue of $DF(0)+\lambda_{cyl} Dg(0)$.
\end{lem}

\begin{proof} We write
\begin{align*}
\partial_\lambda u_{cyl}^{\lambda_{cyl}} = \sum_{j=0}^\infty b_j \varphi_j 
\end{align*}
with suitable $b_j \in \mathbb{R}$ and (unconditional) convergence in $W^2_2(-1,1) \hookrightarrow C^1 \big([-1,1]\big)$. We have convergence in $W^2_2(-1,1)$ since $\partial_\lambda u_{cyl}^{\lambda_{cyl}}$ belongs to $W^2_{2,D}(-1,1)$.
We write
\begin{align*}
-[DF(0) + \lambda_{cyl} Dg(0) \big]\partial_\lambda u_{cyl}^{\lambda_{cyl}} = \sum_{j=0}^\infty -\mu_j(\sigma) b_j \varphi_j 
\end{align*}
with $\big (\mu_j(\sigma)\big)_j$ denoting the eigenvalues of $DF(0) + \lambda_{cyl} Dg(0)$, which are all strictly smaller than zero. Based on $\eqref{cylabl}$ and Lemma \ref{consteig}, a comparison of Fourier coefficients in $L_2(-1,1)$ yields the assertion.
\end{proof}

Since each cosine in the series \eqref{fseries} is scaled by an odd multiple of $\pi/2$, a sufficient condition for \eqref{fseries} to be positive is presented in Lemma \ref{directioncyl2} in the appendix: If
\begin{align*}
 C_1:=a_0 - \sum_{j=1}^\infty (2j+1) \vert a_j \vert  \color{black}>0 \,
\end{align*}
with coefficients $a_j$ from Lemma \ref{directioncyl3a}, then 
\begin{align}
 \partial_\lambda u^{\lambda_{cyl}}_{cyl} (z) \geq C_1 \frac{4}{\pi \ln(2)^2}\cos \left ( \frac{\pi}{2} \,z \right )\,, \qquad z \in (-1,1)\,. \label{fseriespos}
\end{align}
Inserting the $a_j$'s, we are left with checking convergence and sign of
\begin{align*}
C_1= \frac{1}{(-\mu_0(\sigma))} - \sum_{j=1}^\infty \frac{1}{(-\mu_{2j}(\sigma))}\,.
\end{align*}
We recall that the eigenvalues $\mu_{2j}(\sigma)$ of $DF(0)+\lambda_{cyl} Dg(0)$ can be written as 
\begin{align*}
\mu_{2j}(\sigma)=\mu(\sigma^2\nu_{2j}) 
\end{align*} with eigencurve profile $[s \mapsto \mu(s)]$ and the eigenvalues of the Dirichlet-Laplacian 
 $$\nu_{2j} = \frac{(2j+1)^2}{4} \pi^2 \,, \qquad j\in \mathbb{N}\,.$$
Upper and lower bounds for the eigenvalues $\mu_{2j}(\sigma)$ are derived from properties of the eigencurve profile $[s \mapsto \mu(s)]$:

\begin{lem}\label{directioncyl3}
Let $\sigma > \sigma_{cyl}$. Then, the eigenvalues of $DF(0)+ \lambda_{cyl}Dg(0)$ satisfy
\begin{align*}
-\sigma^2  \frac{(2j+1)^2}{4} \,\pi^2 < \mu_{2j}(\sigma) <- \frac{3}{10} \frac{\pi^2}{4} \, \sigma^2 \, \big ((2j+1)^2-1\big) \,, \qquad j \in \mathbb{N}.
\end{align*}
\end{lem}

\begin{proof} {\bf(i)} {\it We derive the lower bound}: For $s \in [0,\infty)$, we have 
\begin{align*}
\mu(s) = -s +3+2 \partial_r h_s(1) 
\end{align*}
with $h_s$ being the solution to \eqref{eqCyl13a}. Because $[s \mapsto \partial_r h_s(1)]$ is an increasing function by Remark \ref{cylincreasing} and $\partial_r h_0(1)= -2 +1/\ln(2)$, see \eqref{eqCyl14aa}, we deduce that
\begin{align*}
\mu(s) \geq -s +3 +2 \partial_r h_0(1) = -s + \frac{2}{\ln(2)} -1 > -s \,.
\end{align*}
Inserting $s=\sigma^2 \nu_{2j}$ results in the estimate from below. \\
{\bf(ii)} {\it We derive the upper bound}: Since $\sigma > \sigma_{cyl}$, we find $\sigma^2 \nu_{2j} > s_0$ with $s_0$ being the unique zero of the eigencurve profile $[s \mapsto \mu(s)]$. Because $[s \mapsto \mu^\prime(s)]$ is decreasing on $[0,\infty)$, see the proof of Proposition \ref{Cyl4}, with $\mu^{\prime}(0)< -3/10$ by Lemma \ref{nu0andnu0prime}, it follows that 
\begin{align*}
\mu_{2j}(\sigma)=\mu(\sigma^2 \nu_{2j}) = \int_{s_0}^{\sigma^2 \nu_{2j}} \mu^\prime(\tilde{s}) \,\mathrm{d} \tilde{s} \leq \frac{-3}{10} (\sigma^2 \nu_{2j}-s_0) \leq \frac{-3}{10} \sigma^2 (\nu_{2j}-\nu_0) \,.
\end{align*}
Inserting the expressions for $\nu_{2j}$ and $\nu_0$ concludes the proof.
\end{proof}

Now, we check the sign of $C_1$ in \eqref{fseriespos}. 
\begin{prop}\label{cylpositive}
For $\sigma > \sigma_{cyl}$, equation \eqref{fseriespos} holds with $C_1>0$. In particular, $\partial_\lambda u^{\lambda_{cyl}}_{cyl}(z) > 0$ for each $z \in (-1,1)$, and $\partial_z[\partial_\lambda u^{\lambda_{cyl}}_{cyl}](-1)>0$ as well as $\partial_z[\partial_\lambda u^{\lambda_{cyl}}_{cyl}](1)<0$.
\end{prop}

\begin{proof}
Due to Lemma \ref{directioncyl3}, we have 
$$0<\frac{1}{\big(-\mu_{2j}(\sigma)\big)} \leq \frac{10}{3}\,\frac{4}{\pi^2 \sigma^2} \frac{1}{(2j+1)^2-1} \,, \qquad j \in \mathbb{N} \,,$$
where the infinite sum over the right-hand side converges. Consequently, the constant $C_1$ is finite. Moreover, a second application of Lemma \ref{directioncyl3} ensures that
\begin{align*}
C_1 &=\frac{1}{(-\mu_0(\sigma))} - \sum_{j=1}^\infty \frac{1}{(-\mu_{2j}(\sigma))}
\\
&\geq \frac{4}{\pi^2\sigma^2} \left (  1 -\frac{10}{3} \sum_{j=1}^\infty \frac{1}{\big( (2j+1)^2-1 \big)}                 \right ) \\
&= \frac{4}{\pi^2\sigma^2} \left ( 1- \frac{10}{3} \cdot \frac{1}{4} \right ) =\frac{2}{3\pi^2\sigma^2}>0\,,
\end{align*}
where
\begin{align*}
\sum_{j=1}^\infty \frac{1}{\big( (2j+1)^2-1 \big)} = \lim\limits_{n \rightarrow \infty} \frac{1}{4} \cdot \frac{n}{n+1} = \frac{1}{4}\,.
\end{align*}
Now, an application of Lemma \ref{directioncyl2} concludes the proof.
\end{proof}

{\bf Proof of Theorem \ref{cyldirectionofdeflection}.}
As outlined in \eqref{cylab0}-\eqref{cylabl}, we rely on the linear approximation of $[\lambda \mapsto u^\lambda_{cyl}]$ around $\lambda_{cyl}$, and therefore the properties of the first order term $\partial_{\lambda} u^{\lambda_{cyl}}_{cyl}$ are important: From Proposition \ref{cylpositive}, it follows that  $\partial_\lambda u^{\lambda_{cyl}}_{cyl}$ is positive and that $\partial_z[\partial_\lambda u_{cyl}^{\lambda_{cyl}} ](1)< 0$ as well as $\partial_z[\partial_\lambda u_{cyl}^{\lambda_{cyl}}](-1)>0$. Thanks to the embedding of $W^2_{q,D}(-1,1)$ in $C^1([-1,1])$\,, we find $\varepsilon > 0$ such that 
\begin{align}
\partial_z[\partial_\lambda u_{cyl}^{\lambda_{cyl}} ] (z) &\leq -4 \varepsilon \,, \qquad z \in (1-\varepsilon,1]\,, \notag \\
\partial_z[\partial_\lambda u_{cyl}^{\lambda_{cyl}} ] (z) &\geq\, 4 \varepsilon \,,\ \  \qquad z \in [-1,-1+\varepsilon)\,. \label{Auslenkung6}
\end{align}
Furthermore, since $\partial_\lambda u_{cyl}^{\lambda_{cyl}}$ is continuous and strictly positive on $[-1+\varepsilon,1-\varepsilon]$, we find $\tilde{\varepsilon} > 0$ such that
\begin{align}
\partial_\lambda u^{\lambda_{cyl}}_{cyl}(z) \geq 4 \tilde{\varepsilon} \,, \qquad z \in [-1+\varepsilon,1-\varepsilon]\,. \label{Auslenkung7}
\end{align}
Finally, the continuity of $\big [(z,\lambda) \rightarrow \partial_\lambda u^\lambda_{cyl}(z)\big ]$ and $\big [(z,\lambda) \rightarrow \partial_z[\partial_\lambda u^\lambda_{cyl}](z)\big ]$ allows us to extend \eqref{Auslenkung6} and \eqref{Auslenkung7} to 
\begin{align}
\partial_z[\partial_\lambda u_{cyl}^\lambda ] (z) &\leq - 2\varepsilon \,, \qquad z \in (1-\varepsilon,1] \,, \qquad \   \vert\lambda -\lambda_{cyl}\vert \leq \delta\,,\notag \\
\partial_z[\partial_\lambda u_{cyl}^\lambda ] (z) &\geq\,  2\varepsilon \,,\ \  \qquad z \in [-1,-1+\varepsilon)\,, \ \ \, \vert\lambda -\lambda_{cyl}\vert \leq \delta\,, \label{Auslenkung8}
\end{align}
and 
\begin{align}
\partial_\lambda u^\lambda_{cyl}(z) \geq  2\tilde{\varepsilon} \,, \qquad z \in [-1+\varepsilon,1-\varepsilon]\,, \qquad  \vert\lambda -\lambda_{cyl}\vert \leq \delta\,, \label{Auslenkung9}
\end{align}
for suitably chosen $\delta>0$. Let us now turn to the linear approximation
\begin{align}
u^\lambda_{cyl} = u^{\overline{\lambda}}_{cyl} + \partial_\lambda u_{cyl}^{\overline{\lambda}}\, (\lambda-\overline{\lambda}) + R(\lambda,\overline{\lambda}) \label{tail}
\end{align}
in $W^2_{q,D}(-1,1) \hookrightarrow C^1\big ([-1,1]\big)$ with error term 
\begin{align*}
R(\lambda, \overline{\lambda} ):= \int_0^1 (1-t) \, \partial_\lambda^2\, u^{\overline{\lambda} +t(\lambda-\overline{\lambda})}_{cyl}  \,\mathrm{d}t \,(\lambda-\overline{\lambda})^2
\end{align*}
satisfying the uniform estimate
\begin{align*}
\frac{\ \ \Vert R(\lambda,\overline{\lambda})\Vert_{C^1}}{\vert \lambda-\overline{\lambda} \vert} \leq C \,\vert \lambda - \overline{\lambda} \vert 
\end{align*}
for some $C > 0$ independent of $\lambda\,, \, \overline{\lambda} \in [\lambda_{cyl}-\delta,\lambda_{cyl}+ \delta]$. As a consequence, we can adjust $\delta>0$ such that 
\begin{align}
\frac{\ \ \Vert R(\lambda,\overline{\lambda})\Vert_{C^1}}{\vert \lambda-\overline{\lambda} \vert} \leq \mathrm{min} \big \lbrace \varepsilon, \tilde{\varepsilon} \big \rbrace \,, \qquad 0 < \lambda- \overline{\lambda}  \leq \delta\,, \quad \vert \lambda -\lambda_{cyl}\vert \leq \delta \,. \label{Auslenkung10}
\end{align}
From \eqref{Auslenkung9}-\eqref{Auslenkung10}, it follows that 
\begin{align*}
\frac{u^\lambda_{cyl}(z)-u_{cyl}^{\overline{\lambda}}(z)}{\lambda -\overline{\lambda}}\geq \tilde{\varepsilon}\,, \qquad z \in [-1+\varepsilon, 1-\varepsilon]\,, 
\end{align*}
while \eqref{Auslenkung8} - \eqref{Auslenkung10} yield
\begin{align*}
\frac{\partial_z u^\lambda_{cyl}(z)-\partial_z u_{cyl}^{\overline{\lambda}}(z)}{\lambda -\overline{\lambda}}\geq \varepsilon  \,, \qquad z \in [-1,-1+\varepsilon)\,,
\end{align*}
as well as
\begin{align*}
\frac{\partial_z u^\lambda_{cyl}(z)-\partial_z u_{cyl}^{\overline{\lambda}}(z)}{\lambda -\overline{\lambda}}\leq -\varepsilon  \,, \qquad z \in (1-\varepsilon,1]\,.
\end{align*}
Here, all three estimates above hold for $0 < \lambda- \overline{\lambda}  \leq \delta$ and $\vert\lambda - \lambda_{cyl} \vert \leq \delta$. From these estimates and the fact that $$u^\lambda_{cyl}(\pm 1) = u^{\overline{\lambda}}_{cyl}(\pm 1)=0\,,$$ we deduce 
\begin{align}
u^\lambda_{cyl}(z) > u^{\overline{\lambda}}_{cyl}(z) \,, \qquad z \in (-1,1) \,, \quad 0 < \lambda- \overline{\lambda}  \leq \delta\,, \quad \vert \lambda-\lambda_{cyl} \vert \leq \delta\,, \label{Auslenkung11}
\end{align}
from which the statement follows. 
\qed

\end{subsection}
\end{section}

\begin{section}{Discussion on Parameter Relations}\label{Disc}
Finally, we compare the values for $\sigma_{crit}$ and $\sigma_{cyl}$. The outcome points towards a balancing effect of the electrostatic force confirming observations from the simplified model in \cite{Moulton08}. Recall that, in general, $\sigma$ gives the ratio between the radius of the metal rings between which the soap film is spanned divided by half of their distance. The first characteristic value of $\sigma$ is the critical ratio $\sigma_{crit} \approx 1.5$ from \eqref{sigmacrit} below which no stationary solutions to \eqref{eq1}-\eqref{eq3} with $\lambda = 0$ exist. From \cite{DK91} and the parabolic comparison principle, it even follows that for $\sigma < \sigma_{crit}$ there exists no global solution for the film's dynamics at all. The second characteristic value of $\sigma$ is $\sigma_{cyl}$ from \eqref{thv} at which the stability of the cylinder and also that of $u^\lambda_{cyl}$ switches. The value is given approximately by
\begin{align*}
 \sigma_{cyl} = \frac{2 \sqrt{s_0}}{\pi} \leq \frac{2 \,\sqrt{4.2}}{\pi} \approx 1.3 \,.
 \end{align*}
Herein, the inequality is formally justified by a numerical plot, see Figure \ref{bildeigprofile}, which shows that the unique zero $s_0$ of the eigencurve profile $[s \mapsto \mu(s)]$ satisfies $s_0 \leq 4.2$. In summary, we have $\sigma_{cyl} < \sigma_{crit}$. Hence, for $\sigma \in (\sigma_{cyl}, \sigma_{crit})$ and
suitably scaled electrostatic force, we observe stability of the cylinder by Theorem \ref{stability2} -- in particular, many solutions to the dynamical version of \eqref{eq1}-\eqref{eq3},see \eqref{dyn}, with $\lambda$ close to $\lambda_{cyl}$ do exist globally in time --
while in absence of the electrostatic force, i.e. for $\lambda=0$ in the dynamical version of \eqref{eq1}-\eqref{eq3}, see \eqref{dyn}, all solutions cease to exist after a finite time.
In other words, this is a first indication that the electrostatic force might be used to prevent the soap film bridge from pinching off. 
\begin{figure}[!h]
\includegraphics[width=90mm]{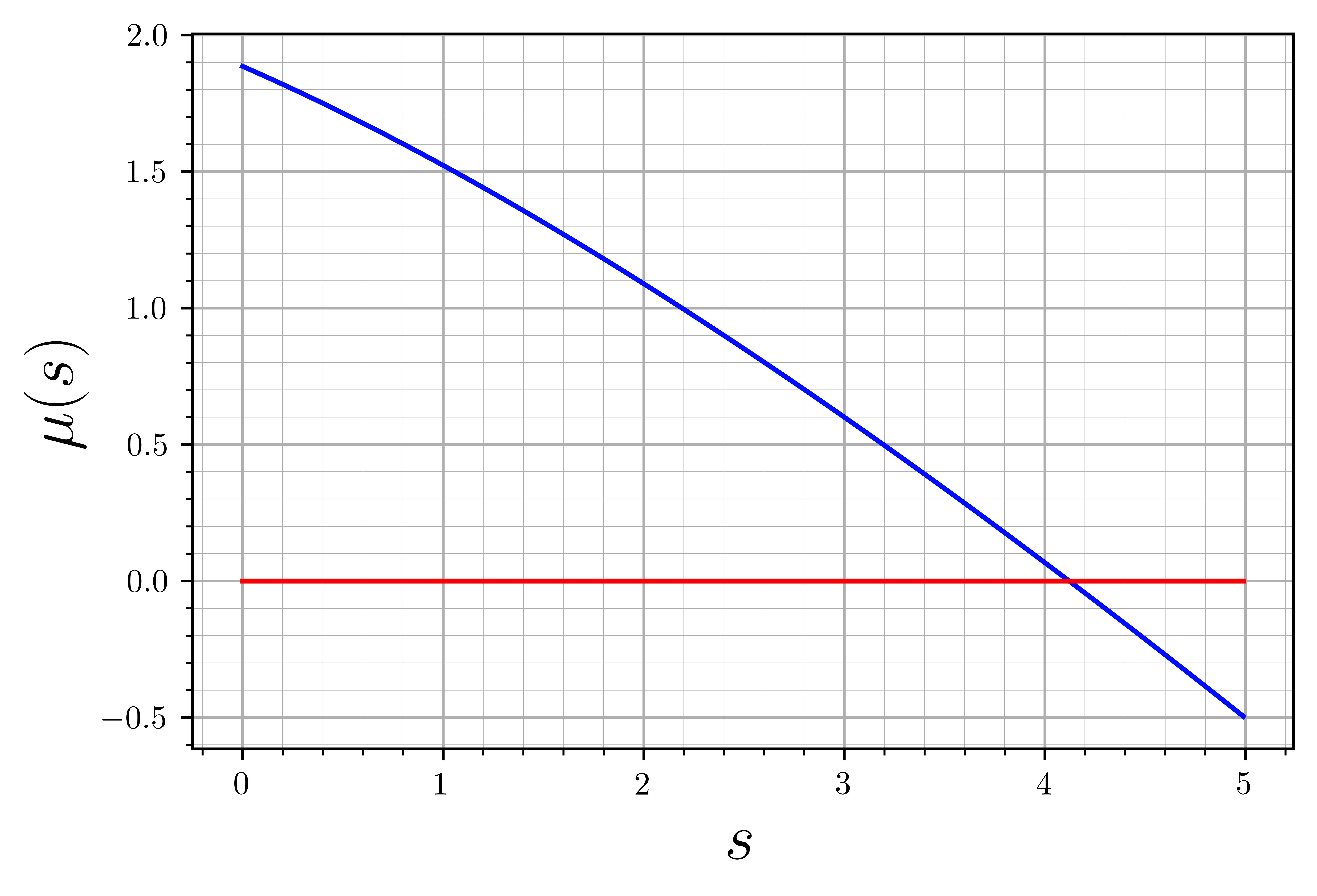}
\caption[Numerical Plot of Eigencurve Profile]{Numerical plot of the eigencurve profile $[s \mapsto \mu(s)]$ (blue) together with the constant $[s \mapsto 0]$ (red), created with Python routine scipy.integrate.solve\_bvp. 
\label{bildeigprofile}}
\end{figure}

\end{section}

\section*{Acknowledgement} This paper contains results and edited text from my PhD-thesis. I am very thankful to my PhD-supervisor Christoph Walker for his support.

\appendix
\section{}\label{Appendices}

\subsection{Computation of the Linearization}\label{LinearizationApp}

 Here, we give the precise form of the transformed $v$-dependent differential operator $L_v$, which is
 \begin{align}
  L_v w &:= \sigma^2 (1-v) \partial_z^2 w - 2 \sigma^2 \, \partial_z v \,  ( 2-r  ) \partial_r\partial_z w \nonumber \\
  &\quad\ \ +\frac{1+\sigma^2 (\partial_z v)^2  ( 2-r  )^2}{1-v}\, \partial_r^2 w \nonumber \\
  &\quad\ \  + \left [ - \sigma^2  ( 2 - r  ) \Big (  \partial_z^2 v +  \frac{2 (\partial_z v)^2 }{1-v} \Big ) + \frac{1}{2v +(1-v)r } \right ] \partial_r w\,. \label{Lvnondivergence}
 \end{align}
 In divergence form, this operator reads
 \begin{align}
L_v w = \mathrm{div} \left (A(v) \nabla w \right ) + d(v)\cdot \nabla w   \label{Lvdivergence}
 \end{align}
 with
 \begin{align*}
  A(v) = [a_{ij}(v)]_{i,j=1}^2 &:= \begin{pmatrix}
          \sigma^2 (1-v)  & - \sigma^2  \,\partial_z v\, (2-r ) \\
         - \sigma^2 \,\partial_z v \, (2-r )  & \displaystyle\frac{1+\sigma^2 (\partial_z v)^2(2-r)^2}{1-v}
         \end{pmatrix}
\,, \\
d(v) = \begin{pmatrix} d_1(v) \\ d_2(v) \end{pmatrix}&:=\begin{pmatrix}  0\\ \displaystyle\frac{1}{2v+(1-v)r}\end{pmatrix} \,,
 \end{align*}
 and by  \cite[Lemma 3.1]{LSS24a} the operator $-L_v$ is again uniformly elliptic for each $v \in S$ (see \eqref{defS}). \\
\color{black}

{\bf Proof of Lemma \ref{Cyl3a}.}
We have to show that the linearization of $F(u)+\lambda_{cyl} g(u)$ around $u=0$ is given by \eqref{lo}, i.e. by
\begin{align*}
\big (DF(0)+\lambda_{cyl} Dg(0) \big ) v=\sigma^2 \partial_z^2 v +3v + 2 \, \partial_r  (-\Delta_{cyl,D}  )^{-1} \Big [- \frac{2}{r^3} v - \sigma^2\, \frac{2-r}{r}\, v_{zz} \Big ](\, \cdot \,,1) 
\end{align*} 
for $v \in W^2_{q,D}(-1,1)$. Here, we recall that $\lambda_{cyl}= \mathrm{ln}(2)^2$ and refer to \eqref{CatFa} and \eqref{filmdimensionless0} \color{black} for the definitions of $F$ and $g$.\\

{\bf(i)} {\it For $DF(0)$:} Using the relation $$\sigma \partial_z \mathrm{arctan}(\sigma \partial_z u)= \displaystyle\frac{\sigma^2 \partial_z^2 u}{1+\sigma^2(\partial_z u)^2}\,,$$
we compute
\begin{align*}
 DF(u)v  = \frac{\sigma^2 }{(1+\sigma^2u_z^2)} \, v_{zz} &- \frac{2 \sigma^4 u_{z}u_{zz}}{(1+\sigma^2 u_z^2)^{2}}\,v_z + \frac{1}{(u+1)^{2}}\,v ,
\end{align*}
which, evaluated at $u=0$, yields
\begin{align}
DF(0)v = \sigma^2 \partial_z^2 v +v\,.  \label{DF(0)}
\end{align}
{\bf(ii)} {\it For $\lambda_{cyl}Dg(0)$}: It remains to show that the linearization of $g$ around $0$ is given by
\begin{align}
\lambda_{cyl}\, Dg(0)v= 2 v  + 2 \, \partial_r  (-\Delta_{cyl,D}  )^{-1} \Big [- \frac{2}{r^3} v - \sigma^2\, \frac{2-r}{r}\, v_{zz} \Big ](\, \cdot \,,1) \label{Dg(0)}
\end{align}
for $v \in W^2_{q,D}(-1,1)$:

 \color{black}

Using the definition of $g$ from \eqref{filmdimensionless0}, \eqref{eqCyl1b} and the relation $\psi_u=\phi_u\circ T_u$ with $T_u$ defined in \eqref{deftrafoT}, we compute
\begin{align}
Dg(0)v&= D \Big [ (1+\sigma^2 u_z^2)^{3/2}\,\big \vert\partial_r \psi_{u} (z,u+1) \big \vert^2 \Big ] \,\Big\vert_{u=0} v \nonumber \\
&= 2 \,\partial_r \psi_0(z,1)\, D \big [ \partial_r\psi_{u}(z,u+1)\big ]\,\Big \vert_{u=0}v \nonumber \\
&= \frac{2}{\mathrm{ln}(2)}\, D \bigg [ \frac{\partial_r\phi_{u}(z,1)}{1-u} \bigg ]\,\bigg\vert_{u=0}v \nonumber \\
&=\frac{2}{\mathrm{ln}(2)}\, \Big ( \partial_r \phi_0(z,1)\, v + D \big [ \partial_r \phi_{u}(z,1) \big ] \Big \vert_{u=0} v \Big )\,. \label{eqCyl8}
\end{align}
At this point, we recall from \eqref{defphiv} that $\phi_{u}$, the transformation of $\psi_{u}$ to the fixed domain $\Omega$, is given by 
\begin{align*}
\phi_{u} = -L_D(u)^{-1} L_{u} \Big ( \frac{\mathrm{ln}(r)}{\mathrm{ln}(2)} \Big ) + \frac{\mathrm{ln}(r)}{\mathrm{ln}(2)} 
\end{align*}
with $L_{u}$ and $L_{D}(u)$ defined in \eqref{Lvnondivergence} and \eqref{defLDv} respectively. In particular,
\begin{align}
L_0 \Big ( \frac{\ln(r)}{\ln(2)} \Big )=\frac{1}{\ln(2)} \Big (\frac{1}{r} \partial_r \big ( r \partial_r \ln(r) \big ) + \sigma^2 \partial_z^2 \ln(r)\Big)=0\,, \label{eqCyl9}
\end{align}
and hence 
\begin{align*}
\phi_0 =\frac{\ln(r)}{\ln(2)}
\end{align*}
as well as
\begin{align}
   \partial_r \phi_0(z,1)= \frac{1}{\ln(2)} \,.  \label{eqCyl10}
\end{align}
Moreover,
\begin{align*}
L_{u} \Big ( \frac{\ln(r)}{\ln(2)} \Big )= &\frac{1+\sigma^2 u_z^2(2-r)^2}{1-u} \,\partial_r^2 \Big (\frac{\ln(r)}{\ln(2)} \Big )\\
&+ \Big ( - \sigma^2(2-r) u_{zz} - 2 \sigma^2 \frac{2-r}{1-u} \, u_z^2 + \frac{1}{2u+(1-u)r}\Big ) \, \partial_r \Big (\frac{\ln(r)}{\ln(2)}\Big) \\
=& \frac{1}{\ln(2)}\, \bigg [- \frac{1+\sigma^2u_z^2(2-r)^2}{(1-u)} \frac{1}{r^2}\\
&\qquad \quad +\Big ( - \sigma^2(2-r) u_{zz} - 2 \sigma^2 \frac{2-r}{1-u} \, u_z^2 + \frac{1}{2u+(1-u)r}\Big ) \, \frac{1}{r}\, \bigg ]
\end{align*}
implies that
\begin{align}
D L_{u} \Big ( \frac{\ln(r)}{\ln(2)} \Big ) \Big \vert_{u=0} v &= \frac{1}{\ln(2)} \bigg [ - \frac{1}{r^2}\,v  + \Big (- \frac{1}{r^2} \,(2-r)\,v - \sigma^2(2-r)v_{zz} \Big ) \,\frac{1}{r}  \, \bigg ]          \nonumber     \\
&= \frac{1}{\ln(2)} \, \Big [  - \frac{2}{r^3} \, v - \sigma^2 \frac{2-r}{r} \, v_{zz} \,\Big ] \,. \label{eqCyl11}
\end{align}
Combining \eqref{eqCyl9} and \eqref{eqCyl11} results in
\begin{align}
D \phi_{u} \big \vert_{u=0}\, v &= D \Big [ -L_D(u)^{-1} L_{u} \Big ( \frac{\ln(r)}{\ln(2)} \Big ) \Big ] \,\Big \vert_{u=0} v \nonumber \\
&=\frac{1}{\ln(2)} \, \big( -L_D(0) \big)^{-1} \,\Big [- \frac{2}{r^3} \, v - \sigma^2 \frac{2-r}{r} \, v_{zz} \,\Big ] \nonumber \\ 
&=\frac{1}{\ln(2)} \,(-\Delta_{cyl,D})^{-1} \,\Big [- \frac{2}{r^3} \, v - \sigma^2 \frac{2-r}{r} \, v_{zz} \,\Big ]\,. \label{eqCyl12}
\end{align}
For the last line, we used $\big(-L_D(0)\big)^{-1} = (-\Delta_{cyl,D})^{-1}$. Because the chain rule yields
$$D \big [\partial_r \phi_u( \, \cdot \,,1) \big ] \Big \vert_{u=0} v = \partial_r \big ( D \phi_u \big \vert_{u=0} v \big )(\,\cdot \,,1)\,,$$ it follows from \eqref{eqCyl8}, \eqref{eqCyl10} and \eqref{eqCyl12} that 
\begin{align*}
Dg(0)v = \frac{2}{\ln(2)^2} \bigg [\, v +\partial_r (-\Delta_{cyl,D})^{-1}\,\Big [- \frac{2}{r^3} \, v - \sigma^2 \frac{2-r}{r} \, v_{zz} \,\Big ]( \, \cdot \, ,1)\bigg ]\,.
\end{align*}
Since $\lambda_{cyl} = \ln(2)^2$ by \eqref{eqCyl2}, step (ii) and hence the computation of $DF(0)+\lambda_{cyl} Dg(0)$ is completed. \qed \\
\color{black}
 
\subsection{Odd Cosine Sums}\label{OddCosineSums}

We present a sufficient condition for an odd cosine sum to be positive on $(-1,1)$. In the following, a cosine sum is odd if each cosine is scaled by an odd multiple of $\pi/2$. \\

As preparation, we reduce odd cosine sums to even ones by applying trigonometric identities, see \cite[Lemma 5]{Boyd07}. 

\begin{lem}\label{directioncyl1}
Let $n \in \mathbb{N}$ and $a_0, \dots, a_n \in \mathbb{R}$. Consider the sum of odd cosines 
\begin{align*}
f(z):= \sum_{j=0}^n a_j \cos \bigg ( \frac{(2j+1)\pi}{2}\, z \bigg) \,, \qquad z \in (-1,1) \,.
\end{align*}
Then \vspace{-2mm}
\begin{align*}
\frac{f(z)}{\cos  ( \frac{\pi}{2} z )} = \sum_{j=0}^n b_j \cos  ( j\pi z)\,, \qquad z\in(-1,1)\,,
\end{align*}
\vspace{-1mm}
with coefficients 
\begin{align*}
&b_0:= \sum_{k=0}^n (-1)^k a_k \,, \qquad b_j := 2 \sum_{k=j}^n (-1)^{k-j} a_k\,, \qquad j =1, \dots ,n\,.
\end{align*}
\end{lem}

\begin{proof}
This follows by induction. For the step from $n-1$ to $n$, we apply the trigonometric identity 
\begin{align*}
\cos \bigg (\frac{(2n+1) \pi}{2} z \bigg ) + \cos \bigg (\frac{(2n-1) \pi}{2} z \bigg ) = 2 \cos(n\pi z) \cos \Big (\frac{\pi}{2} z \Big )
\end{align*}
to rewrite $f$ as 
\begin{align*}
f(z)&= \sum_{j=0}^{n-2} a_j \cos \bigg ( \frac{(2j+1)\pi}{2}\, z \bigg) + (a_{n-1}-a_n) \cos \bigg ( \frac{(2(n-1)+1)\pi}{2}\, z \bigg) \\
&\ \ \ +2a_n \cos(n\pi z) \cos  \Big (\frac{\pi}{2} z \Big ) \,.
\end{align*}
Consequently, the coefficients $b_0$ to $b_{n-1}$ of $\frac{f(z)}{\cos(\frac{\pi}{2} z)}$ are given by\footnote{For $b_{n-1}$, we use the convention $\sum_{k=n-1}^{n-2} (-1)^{k-(n-1)}a_k =0$.}
\begin{align*}
b_0 &= \sum_{k=0}^{n-2} (-1)^k a_k + (-1)^{n-1} (a_{n-1} -a_n) = \sum_{k=0}^{n} (-1)^k a_k \,,\\
b_j &= 2\sum_{k=j}^{n-2} (-1)^{k-j} a_k + 2(-1)^{n-1-j} (a_{n-1} -a_n) = 2\sum_{k=j}^{n} (-1)^{k-j} a_k \,, \qquad j=1, \dots ,n-1,\\
\end{align*}
while  \vspace{-8mm}
\begin{align*}\vspace{-7mm}
b_n = 2 a_n = 2 \sum_{k=n}^n (-1)^{k-n} a_k
\end{align*}\vspace{-3mm}
is also fulfilled.
\end{proof}

Now, the condition for positivity reads:

\begin{lem}\label{directioncyl2}
Let $n \in \mathbb{N}$ and $a_0 , \dots, a_{n} \in \mathbb{R}$. Moreover, let $C > 0$ with
\begin{align*}
 a_0 - \sum_{j=1}^{n} (2j+1) \,\vert a_j \vert \geq C >0\,.
\end{align*}
Then, the sum of odd cosines 
\begin{align*}
f(z):= \sum_{j=0}^{n}  a_j \cos \bigg ( \frac{(2j+1)\pi}{2}\, z \bigg) \,, \qquad z \in (-1,1) 
\end{align*}
satisfies
$$f(z) \geq  C \,\cos \Big (\frac{\pi}{2} z \Big )\,, \qquad z \in (-1,1)\,.$$
In particular, $f(z)>0$ for each $z \in (-1,1)$ and $f_z(\pm 1) \neq 0$.
\end{lem}

\begin{proof}
We prove the equivalent statement 
\begin{align*}
\frac{f(z)}{\cos(\frac{\pi}{2} z)} \geq C \,, \qquad z \in (-1,1).
\end{align*}
Lemma \ref{directioncyl1} yields
\begin{align*}
\frac{f(z)}{\cos(\frac{\pi}{2} z)} = \sum_{j=0}^{n} b_j \cos(j \pi z) 
\end{align*}
with coefficients
\begin{align*}
b_0= \sum_{k=0}^n  (-1)^k a_k  \,, \quad b_j = 2 \sum_{k=j}^n (-1)^{k-j} a_k \,, \qquad j =1, \dots ,n\,,
\end{align*}
so that
\begin{align*}
\frac{f(z)}{\cos(\frac{\pi}{2} z)}  &\geq b_0 - \sum_{j=1}^n \vert b_j \vert  \geq \sum_{k=0}^n (-1)^k a_k - 2 \sum_{j=1}^n  \sum_{k=j}^n \vert a_k \vert  \\
&\geq a_0 - \sum_{k=1}^n \vert a_k \vert - 2 \sum_{j=1}^n j\,\vert a_j \vert  = a_0 - \sum_{j=1}^n (2j+1) \vert a_j \vert \geq  C\,.
\end{align*}
\end{proof}

\bibliographystyle{siam}
\bibliography{BibliographyDoc}
\end{document}